\documentclass[a4paper,12pt]{amsart}

\usepackage[utf8]{inputenc}
\usepackage{amsmath, amsthm, amsfonts, amssymb}
\usepackage[foot]{amsaddr}

\usepackage{todonotes}
\usepackage{hyperref} %internal references
\usepackage[mathlines, pagewise]{lineno} %for adding line numbering on the margin, useful when submitting for review
 \newtheorem{theorem}{Theorem}[section]
  \newtheorem{corollary}[theorem]{Corollary}
 \newtheorem{remark}[theorem]{Remark}
 \numberwithin{equation}{section}

\usepackage{amsmath}
 \usepackage[mathscr]{euscript}

\title[Estimates  for  the  approximation characteristics]{Estimates  for  the  approximation characteristics 
of the  Nikol'skii-Besov classes  of  functions   with  mixed smoothness  in  the   space $\boldsymbol{B_{q,1}}$}

\author{K.V.~Pozharska$^{1, 2}$, A.S.~Romanyuk$^1$
}

\address{
$^1$Institute of Mathematics of the NAS of Ukraine, Kyiv, Ukraine
$^2$Faculty of Mathematics, Chemnitz University of Technology, 
Germany}

\email{pozharska.k@gmail.com}
\email{romanyuk@imath.kiev.ua}

\allowdisplaybreaks


\begin{document}

\begin{abstract}
Exact-order estimates are obtained for some approximation characteristics of the classes of periodic multivariate functions with  mixed smoothness
(the  Nikol'skii-Besov classes $B^{\boldsymbol{r}}_{p, \theta}$) in  the   space $B_{q,1}$,
$1 \leq p, q \leq \infty$, $1\leq \theta\leq \infty$, whose norm is stronger than the $L_q$-norm. It is shown, that in the multivariate case (in contrast to the univariate)
in most of the considered situations the obtained estimates differ in order from the corresponding estimates in the space $L_q$.

Besides, a significant progress is made in estimates for the considered approximation characteristics of the classes  $B^{\boldsymbol{r}}_{p, \theta}$  in the space $B_{q, 1}$ comparing to the known estimates in the space $L_q$.
\end{abstract}

\maketitle

\setcounter{tocdepth}{1}

%\tableofcontents

\setlength{\parskip}{\medskipamount}
%\newpage

  \section{Introduction}
In the paper, we investigated some approximation characteristics of the  Nikol'skii-Besov classes $B^{\boldsymbol{r}}_{p,\theta}$ of periodic multivariate functions with  mixed smoothness
 in  the   space $B_{q,1}$,
$1 \leq p, q \leq \infty$, $1\leq \theta\leq \infty$.

A motivation for investigating the approximation problems in this particular space
$B_{q,1}$ is given by the fact, that in  $L_q$ some important cases still remain open
(see Open Problems in \cite{Dung_Temlyakov_Ullrich2019,Temlyakov2018} and more detailed remarks in Sections \ref{sec2:App_shFs} and \ref{sec3:widths_en}). 

The space $B_{q,1}$, as a linear subspace  in $L_q$,  $1 \leq q \leq \infty$,  respectively,  has  a peculiarity   in   that  its  norm  is  stronger  than   the  $L_q$-norm.
However, recently it was shown in \cite{Cobos_Kuhn_Sickel_2019} that the
 estimates of approximation numbers of periodic functions from the Sobolev classes in  $L_\infty$  remain valid when we
measure the error in the norm of a ``zero'' Besov space $B_{\infty,1}$. This unexpected result has some practical use, since the Littlewood–Paley characterization of the target spaces simplifies the computation of approximation or other $s$-numbers.  Besides, the authors proved  the chain of continuous embeddings of the associated Wiener algebra into the space $B_{\infty,1}$ and, in turn, into the space of continuous functions where the embedding operators are of norm one. A behavior of entropy numbers of the Triebel-Lizorkin classes, in particular, in the norm of certain  ``zero'' Besov spaces was studied also in \cite{Vybiral2006} (e.g., Theorem 4.11, see also Section 4.6).

   Our findings complement and generalize the results, obtained in the  number  of  papers \cite{Belinsky1998,Fedunyk_Hembarskyi_Hembarska2020, Fedunyk_Hembarska2022,Hembarska_Zaderey2022,Kashin_Temlyakov1994,Romanyuk_Romanyuk2020,  Romanyuk_Romanyuk2021,Romanyuk_Yanchenko2022,Temlyakov1990}, where
  the authors investigate the approximation   of  the Nikol'skii-Besov classes  $B^{\boldsymbol{r}}_{p,\theta}$,  Sobolev classes $W^{\boldsymbol{r}}_{p,\alpha}$ 
  of  periodic  multivariate  functions  with  mixed  smoothness  and   some  their   analogs
  in the space $B_{q,1}$,  $q \in \{ 1, \infty \}$.
     Important is, that we get estimates for the approximation characteristics of the classes
    $B^{\boldsymbol{r}}_{p, \theta}$ in the space $B_{q,1}$  for the case  $1 < q < \infty$.

The paper consists of three parts.

The first part (Section \ref{sec1:FCandS}) plays an auxiliary role. Here we introduce  necessary notation and define the considered function classes and spaces $B_{q,1}$,
 in which we measure the approximation error.

In the second part (Section \ref{sec2:App_shFs}) we obtain exact-order estimates for the approximation of functions from the classes
 $B^{\boldsymbol{r}}_{p,\theta}$ in the space $B_{q,1}$ by their step hyperbolic cross Fourier sums. Besides, we get orders for the best approximation of the indicated function classes in the respective space by trigonometric polynomials with harmonics from the step hyperbolic crosses.

The third part (Section \ref{sec3:widths_en}) is devoted to obtaining the exact-order estimates for the Kolmogorov widths and entropy numbers of the classes
 $B^{\boldsymbol{r}}_{p,\theta}$ in the space $B_{q,1}$ for some relations between the parameters  $p$  and $q$.
Combining these estimates for the Kolmogorov widths and the results of Section \ref{sec2:App_shFs}, we write down the exact-order estimates for the linear widths of these classes in the space $B_{q,1}$.

From the results of our research, we can conclude the following.

A significant progress is made in estimates for the best approximations, widths and entropy numbers
of the classes  $B^{\boldsymbol{r}}_{p, \theta}$  in the space $B_{q,1}$
comparing to the known estimates for the corresponding quantities in the space $L_q$.

  We show, that the estimates for the Kolmogorov and linear widths of the mentioned function classes in the space $B_{q,1}$ are realized by the subspaces of trigonometric polynomials with harmonics from the step hyperbolic crosses of the respective dimension.

We prove, that in most of the cases (for $d\geq 2$) the considered approximation characteristics  of the classes $B^{\boldsymbol{r}}_{p,\theta}$ in the space $B_{q,1}$ differ in order from the corresponding characteristics in the space $L_q$.

 These issues will be discussed more detailed in remarks to the results.

\section{Function classes $B^{\boldsymbol{r}}_{p,\theta}$ and spaces $B_{q,1}$}\label{sec1:FCandS}

 Let $\mathbb{R}^d$  be a $d$-dimensional space with  elements $\boldsymbol{x}=(x_1, \dots, x_d)$ and
 $(\boldsymbol{x}, \boldsymbol{y}) = x_1\,y_1+ \dots + x_d\,y_d$. By $L_p(\mathbb{T}^d)$, $\mathbb{T}^d =  \prod^{d}_{j=1}\,[ 0,\,2\pi)$, we denote the space of
 $2\pi$-periodic in  each variable
 functions  $f(\boldsymbol{x})$,  for  which
\begin{align*}
    \|f\|_p &:=    \|f\|_{{L_p}(\mathbb{T}^d)} =
   \Big((2\pi)^{-d}
  \int_{\mathbb{T}^d}|f(\boldsymbol{x})|^p\,d{\boldsymbol{x}}\Big)^{1/p} <\infty,\quad 1 \leq p < \infty,
\\
 \|f\|_\infty &:=  \|f\|_{{L_\infty}(\mathbb{T}^d)}  =
 \operatorname{ess\,sup}\limits_{{\boldsymbol{x}}\in  {\mathbb{T}^d}}|f(\boldsymbol{x})|< \infty.
  \end{align*}

We put
  $$
 L^0_p := L^0_p(\mathbb{T}^d) := \{
 %f\colon \ 
 f \in L_p(\mathbb{T}^d)\colon \int^{2\pi}_0 f(\boldsymbol{x})dx_{j} = 0 \ \ \text{a.e.}, \, j = 1, \dots, d \}.
  $$

  For a function $f \in L^0_{p}$, $1 \leq p \leq \infty$,  we  consider its  first  difference
   in the  $j${th}   variable  with step  $h \in \mathbb{R}$:
    $$
\Delta_{h,j} f(\boldsymbol{x}) = f(x_1, \ldots, x_{j-1}, x_j + h, x_{j+1}, \ldots, x_d)  - f(\boldsymbol{x})
$$
 and  define its  $l$th difference, $l \in \mathbb{N}$,  by
 $$
\Delta^l_{h,j} f(\boldsymbol{x}) = \overbrace{\Delta_{h,j} \cdots \Delta_{h,j}}\limits^l f(\boldsymbol{x}).
 $$

Assume that the  vectors $\boldsymbol{k} 
%= (k_1, \ldots, k_d)
\in \mathbb{N}^d$
%, $k_j \in \mathbb{N}$,  
and  $\boldsymbol{h}
%= (h_1, \ldots, h_d)
\in \mathbb{R}^d$
%$h_j \in \mathbb{R}$, $j = \overline{1, d}$,
are given.  Then the mixed difference  of  order $\boldsymbol{k}$  with  a vector step $\boldsymbol{h}$  is   defined  by  the  equality
$$
 \Delta^{\boldsymbol{k}}_{\boldsymbol{h}} f(\boldsymbol{x}) = \Delta^{k_1}_{h_1, 1} \, \Delta^{k_2}_{h_2, 2} \cdots\, \Delta^{k_d}_{h_d, d}\,  f(\boldsymbol{x}).
$$

The  spaces $B^{\boldsymbol{r}}_{p,\theta}(\mathbb{T}^d)$, $1 \leq p,\theta \leq \infty$,  where 
%$r$  is  a given vector with nonnegative coordinates 
$\boldsymbol{r} 
%= (r_1, \ldots, r_d)
\in \mathbb{R}^d$ is  a given vector with the elements $r_j > 0$, $j = 1,\dots,  d$,   are   defined  as  follows:
$$
B^{\boldsymbol{r}}_{p,\theta} := B^{\boldsymbol{r}}_{p,\theta}(\mathbb{T}^d):=  
\{f \in L^0_p\colon \ \|f\|_{B^{\boldsymbol{r}}_{p,\theta}} < \infty \},
$$
and  the respective norm  is  specified  by  
\begin{align}
\|f\|_{B^{\boldsymbol{r}}_{p,\theta}} &:= \Bigg( \int_{\mathbb{T}^d} \|\Delta^{\boldsymbol{k}}_{\boldsymbol{h}} f \|^{\theta}_p\,
\prod\limits^d_{j=1}\frac{d h_j}{h^{1 + r_j{\theta} }_j} \Bigg)^{\frac{1}{\theta}}, \qquad 1 \leq \theta < \infty,
\label{norm_intro}
\\
\|f\|_{H^{\boldsymbol{r}}_p} &\equiv  \|f\|_{B^{\boldsymbol{r}}_{p,\infty}} :=  \sup\limits_{\boldsymbol{h}} \|\Delta^{\boldsymbol{k}}_{\boldsymbol{h}} f\|_p \prod\limits^d_{j=1}h^{-r_j}_j.
\label{norm_intro_infty}
\end{align}

Here  we  assume  that  the  components of the vectors  $\boldsymbol{k}$
%= (k_1, \ldots, k_d)$  
and $\boldsymbol{r}$
%= (r_1, \ldots, r_d)$ 
satisfy  the  conditions  $k_j > r_j$, $1,\dots,  d$.

In this  form,   the  definition of the  spaces $B^{\boldsymbol{r}}_{p,\theta}$, $1\leq\theta<\infty$,  was    given    by  V.N. Temlyakov \cite{Temlyakov1989_121p}  and,  for  $H^{\boldsymbol{r}}_p \equiv B^{\boldsymbol{r}}_{p,\infty}$, by S.M. Nikol'skii  and  P.I.~Lizorkin \cite{Lizorkin_Nikol'skii1989}.  These  spaces  belong  to the  scale   of  spaces  of  mixed  smoothness, introduced  by  S.M. Nikol'skii \cite{Nikol'skii1963}   and  T.I. Amanov \cite{Amanov1965}.  In addition, they   are   generalizations  of the  well-known  isotropic  Besov  spaces \cite{Besov1961},   and  the  Nikol'skii spaces \cite{Nikol'skii1951}  for  the  case $\theta = \infty$.

%The  class $B^{\boldsymbol{r}}_{p,\theta}$  is  defined  as a  set  of  functions  $f \in L^0_p$  such  that  $\| f\|_{B^{\boldsymbol{r}}_{p,\theta}} \leq 1$. 

In what follows, we introduce an equivalent 
norm representation 
%definition 
for the spaces $B^{\boldsymbol{r}}_{p,\theta}$, which  is  more  convenient   for  calculations.

For  all  vectors  
%$\boldsymbol{k} = (k_1, \ldots, k_d)\in \mathbb{Z}^d$
%$k_j \in \mathbb{Z}$, and  
${\boldsymbol{s}}
%= (s_1, \ldots, s_d)
\in \mathbb{N}^d$,
%$, $s_j \in \mathbb{N}$, $j = \overline{1,d}$,
we  set
\begin{equation}\label{rho_s}
\rho(\boldsymbol{s}) := \{
%k \colon  
\boldsymbol{k}
%= (k_1, \ldots, k_d)
\in\mathbb{Z}^d \colon \ 2^{s_j - 1} \leq | k_j | < 2^{s_j}, \, j = 1, \dots, d \}.
\end{equation}
For $f \in L^0_p$,   let 
\begin{equation}\label{delta_s}
 \delta_{\boldsymbol{s}}(f) := \delta_{\boldsymbol{s}}(f,\boldsymbol{x}) = \sum\limits_{\boldsymbol{k} \in \rho({\boldsymbol{s}})}\, \widehat{f}(\boldsymbol{k}) e^{i(\boldsymbol{k},\boldsymbol{x})},   
\end{equation}
where  $\widehat{f}(\boldsymbol{k}) = \int_{\mathbb{T}^d}f({\boldsymbol{t}}) e^{-i(\boldsymbol{k},{\boldsymbol{t}})} d{\boldsymbol{t}}$  are    the  Fourier  coefficients   of  the    function  $f$.

First,  let  $1 < p < \infty$.
  Then  we can   define  the  norm of the spaces  $B^{\boldsymbol{r}}_{p,\theta}$,
  $\boldsymbol{r} \in \mathbb{R}^d$,   $r_j > 0$, $j = 1, \dots, d$,
 as  follows (see, e.g., \cite{Lizorkin_Nikol'skii1989}):
 %$$
 %B^{\boldsymbol{r}}_{p,\theta} :=  \{  f \in L^0_p\colon \quad %\|f\|_{B^{\boldsymbol{r}}_{p,\theta}} \leq 1 \},
 %$$
%where
\begin{align}\label{eq1}
\|f\|_{B^{\boldsymbol{r}}_{p,\theta}} & \asymp \Bigg( \sum\limits_{{\boldsymbol{s}} \in \mathbb{N}^d} 2^{({\boldsymbol{s}},{\boldsymbol{r}})\theta} \|\delta_{\boldsymbol{s}} (f)\|^{\theta}_p  \Bigg)^{\frac{1}{\theta}}, \quad
  1 \leq \theta < \infty,
\\
\|f\|_{H^{\boldsymbol{r}}_p} & \equiv  \|f\|_{B^{\boldsymbol{r}}_{p,\infty}}\asymp \sup\limits_{{\boldsymbol{s}} \in \mathbb{N}^d} 2^{({\boldsymbol{s}},{\boldsymbol{r}})} \|\delta_{\boldsymbol{s}} (f)\|_p.
 \label{eq1_infty}
\end{align}

Here and  below,  for two non-negative  sequences $\{a(n)\}_{n=1}^\infty$  and $\{b(n)\}_{n=1}^\infty$,   the  notation (order equality)  $a(n) \asymp b(n)$ means  that  there  exist   $C_1, C_2>0$
 that  do  not  depend  on  $n$   and  such  that $C_1 a(n) \leq b(n)$
 (in  this  case,  we  write  $a(n) \ll b(n)$) and  $C_2 a(n) \geq b(n)$   (denoted by $a(n) \gg b(n)$).

 Note  that  after  the  corresponding  modification  of  the ``blocks''  $\delta_{\boldsymbol{s}}(f)$,  the given above
 norm representation  for  the spaces  $B^{\boldsymbol{r}}_{p,\theta}$  can  also  be   generalized  to the  extreme  values  $p=1$  and  $p=\infty$
 (see, e.g., \cite[Remark 2.1]{Lizorkin_Nikol'skii1989}).

 Let   $V_l(t)$, $t\in \mathbb{R}$, $l \in \mathbb{N}$, denotes the  de la Vall\'ee-Poussin kernel  of  the  form
$$
V_l(t) = 1 + 2 \sum\limits^l_{k=1} \cos k t  + 2\sum\limits^{2l - 1}_{k=l+1}\Big( 1 - \frac{k - l}{l} \Big)\cos k t \,,
$$
where for $l=1$ we assume that the third term equals to zero.

We associate  each  vector  ${\boldsymbol{s}}
%= (s_1, \ldots, s_d)
\in \mathbb{N}^d$
%$, $s_j \in \mathbb{N}$, $j = \overline{1,d}$,
with  the  polynomial
$$
A_{\boldsymbol{s}}(\boldsymbol{x}) =  \prod\limits^d_{j=1}(V_{2^{s_j}}(x_j) - V_{2^{s_j - 1}}(x_j)),
$$
and  for  $f \in L^0_p$, $1 \leq p \leq \infty$,     set
$$
A_{\boldsymbol{s}}(f) := A_{\boldsymbol{s}}(f,\boldsymbol{x}) := (f \ast A_{\boldsymbol{s}})(\boldsymbol{x}),
$$
where ``$\ast$''   denotes  the  operation  of  convolution.   Thus, for $1 \leq p \leq \infty$,  the 
norm of the spaces $B^{\boldsymbol{r}}_{p, \theta}$, $\boldsymbol{r} \in \mathbb{R}^d$,   $r_j > 0$, $j = 1, \dots, d$,
can be defined  as follows:
\begin{align}
%B^{\boldsymbol{r}}_{p, \theta} & = \{ f\colon \,
\|f\|_{B^{\boldsymbol{r}}_{p,\theta}} &\asymp \Bigg( \sum\limits_{{\boldsymbol{s}} \in \mathbb{N}^d}
2^{({\boldsymbol{s}},{\boldsymbol{r}})\theta} \, \|A_{\boldsymbol{s}}(f)\|^\theta_p \Bigg)^{\frac{1}{\theta}}, 
%\leq 1
 \quad 1 \leq \theta < \infty, 
 \label{norm_via_A_s_notinfty}
%\}
\\
%B^{\boldsymbol{r}}_{p, \infty} & = \{ f\colon \, 
\|f\|_{H^{\boldsymbol{r}}_p} & \equiv \|f\|_{B^{\boldsymbol{r}}_{p,\infty}} \asymp \sup\limits_{{\boldsymbol{s}} \in \mathbb{N}^d}\, 2^{({\boldsymbol{s}},{\boldsymbol{r}})} \|A_{\boldsymbol{s}}(f)\|_p. 
%\leq 1 \}.
\label{norm_via_A_s_infty}
\end{align}

We will keep the notation $B^{\boldsymbol{r}}_{p,\theta}$ also for the respective classes (the unit balls in the  spaces $B^{\boldsymbol{r}}_{p,\theta}$). For the classes, we will use the respective definition of the norm from 
(\ref{eq1}), (\ref{eq1_infty}) or (\ref{norm_via_A_s_notinfty}), (\ref{norm_via_A_s_infty}) depending on the value of the parameter~$p$. This should not create any confusion, since we study the asymptotic characteristics of the classes of functions.  These definitions of the norm, as noted, are equivalent to (\ref{norm_intro}) and (\ref{norm_intro_infty}).

For forerunners in the investigation of different  approximation characteristics  of  the  Nikol'skii  and  Nikol'skii-Besov  classes  of  periodic  functions  we refer to
\cite{Dung_Temlyakov_Ullrich2019,Romanyuk2012,Temlyakov1989_121p,Temlyakov1993,Temlyakov2018}, and the references therein. Here the authors comment mainly on the respective asymptotic behaviour. It is worth noting that in some situations also preasymptotic bounds are known, see \cite{Cobos_Kuhn_Sickel_2019}
for the results on the Besov classes of functions, \cite{Dung_Nguyen_2021} for the H\"{o}lder-Nikol’skii classes 
and 
\cite{Chen_Wang2017,Dung_Ullrich2013,Kuehn_Mayer_Ullrich_2016,Kuehn_Sickel_Ullrich2015,Kuehn_Sickel_Ullrich2021} for the Sobolev classes. We want to highlight that in this paper only asymptotic behaviour is studied. 

Note, that with a growth of the parameter $\theta$ the classes  $B^{\boldsymbol{r}}_{p, \theta}$  are expanding, i.e.,
$$
B^{\boldsymbol{r}}_{p,1} \subset B^{\boldsymbol{r}}_{p, \theta_1} \subset B^{\boldsymbol{r}}_{p, \theta_2} \subset B^{\boldsymbol{r}}_{p, \infty} \equiv H^{\boldsymbol{r}}_p, \quad 1 \leq \theta_1 \leq \theta_2 \leq \infty.
$$

We assume that coordinates of the vector
 ${\boldsymbol{r}}
 %= (r_1, \ldots, r_d)
\in  \mathbb{R}^d$, as the parameter of the defined classes, are ordered such that
  $0 < r_1 = r_2 = \dots  =  r_\nu < r_{\nu +1} \leq \dots \leq r_d$, and also that ${\boldsymbol{{\boldsymbol{\gamma}}}} 
  %= ({\gamma}_1, \ldots, \gamma}_d)
\in  \mathbb{R}^d  $ is a vector with the coordinates ${\gamma}_j = r_j/r_1$, 
  $j = 1,\dots, d$. Besides, ${\boldsymbol{{\boldsymbol{\gamma}}}}' 
  %= ({\gamma}'_1, \ldots, \gamma'_d)
\in  \mathbb{R}^d  $,  where $\gamma'_j = \gamma_j=1$  if  $j = 1,\dots, \nu$  and $1 < \gamma'_j < \gamma_j$  if 
  %$j = \overline{\nu +1, d }$.
$j = \nu +1, \dots, d$.

In what follows, we define the norm $\|\cdot\|_{B_{q,1}}$ in the subspaces $B_{q,1}\subset L^0_q$,  $1 \leq q \leq \infty$.

For trigonometric polynomials $t$ with respect to the trigonometric system  $\{ e^{i(\boldsymbol{k},\boldsymbol{x})}\}_{\boldsymbol{k} \in \mathbb{Z}^d}$, the norm 
$
\|t\|_{B_{q,1}} := \sum_{{\boldsymbol{s}} \in \mathbb{N}^d} \|A_{\boldsymbol{s}}(t)\|_q$, $1 \leq  q \leq \infty$ (the sum contains a finite number of terms).

Similarly we define the norm 
 for functions $f \in L^0_q$, such that the series
 $\sum_{{\boldsymbol{s}} \in \mathbb{N}^d}\|A_{\boldsymbol{s}}(f)\|_q$  is convergent:
%$\|f\|_{B_{q,1}}$,  
\begin{equation}\label{norm_Bq1_via_As}
    \|f\|_{B_{q,1}} := \sum\limits_{{\boldsymbol{s}} \in \mathbb{N}^d} \|A_{\boldsymbol{s}}(f)\|_q, \quad 1 \leq  q \leq \infty.
\end{equation}

Note, that in the case $1 < q < \infty$ it holds
\begin{equation}\label{norm_Bq1}
  \|f\|_{B_{q,1}} \asymp \sum\limits_{{\boldsymbol{s}} \in \mathbb{N}^d}\|\delta_{\boldsymbol{s}}(f)\|_q.  
\end{equation}

For $f \in B_{q,1}$, $1 \leq q \leq \infty$, it holds
$$
\|f\|_q \ll \|f\|_{B_{q,1}}; \qquad \|f\|_{B_{1,1}} \ll \|f\|_{B_{q,1}} \ll \|f\|_{B_{\infty, 1}}.
$$

The space $B_{q,1}$ is sometimes called the ``zero'' Besov space (when the smoothness parameter $\boldsymbol{r}$ equals $\boldsymbol{0}=(0,\dots,0)\in \mathbb{R}^d$). Besides, it was shown in \cite{Cobos_Kuhn_Sickel_2019} that the space $B_{\infty,1}$ can replace the space $L_\infty$ without changing the associated approximation numbers for the embeddings of the periodic isotropic   Sobolev spaces and Sobolev spaces of functions with dominating mixed smoothness.

\section{Approximation by the step hyperbolic Fourier sums and \newline  best approximation}\label{sec2:App_shFs}

First, we give  definitions of the approximation characteristics, which are investigated in this part of the paper.

For   $n \in \mathbb{N}$, ${\boldsymbol{s}} \in \mathbb{N}^d$  and ${\boldsymbol{\gamma}}
%= ({\boldsymbol{\gamma}}_1, \ldots, {\boldsymbol{\gamma}}_d)
\in \mathbb{R}^d$,
 $\gamma_j > 0$, $j = 1,\dots, d$,  we put
$$
Q^{\boldsymbol{\gamma}}_n := \bigcup\limits_{({\boldsymbol{s}},{\boldsymbol{\gamma}}) < n}\, \rho({\boldsymbol{s}}),
$$
where $\rho({\boldsymbol{s}})$ is defined by (\ref{rho_s}).
The set $Q^{{{\boldsymbol{\gamma}}}}_n$ is called the step hyperbolic cross.
In the case ${{\boldsymbol{\gamma}}}= (1, \ldots, 1) :=\boldsymbol{1} \in \mathbb{N}^d$, we use the notation $Q^{\boldsymbol{1}}_n$.
%instead of $Q^{{\boldsymbol{{\boldsymbol{\gamma}}}}}_n$.

We consider the set of trigonometric polynomials
$$
T(Q^{{\boldsymbol{\gamma}}}_n):= \Big\{ t\colon \ t(\boldsymbol{x}) = \sum\limits_{\boldsymbol{k} \in Q^{{\boldsymbol{\gamma}}}_n} c_{\boldsymbol{k}} e^{i(\boldsymbol{k},\boldsymbol{x})}, \, c_{\boldsymbol{k}} \in \mathbb{C}, \ \boldsymbol{x} \in \mathbb{R}^d \Big\}
$$
and
for $f \in L^0_1$ define
$$
S_{Q^{{\boldsymbol{\gamma}}}_n}(f) := S_{Q^{{\boldsymbol{\gamma}}}_n}(f,\boldsymbol{x}) := \sum\limits_{\boldsymbol{k} \in Q^{{\boldsymbol{\gamma}}}_n} \widehat{f}(\boldsymbol{k})\, e^{i(\boldsymbol{k},\boldsymbol{x})}, \quad
 \boldsymbol{x} \in \mathbb{R}^d,
$$
where, recall, $\widehat{f}(\boldsymbol{k})$ are the Fourier coefficients of the function $f$. 

The polynomial $S_{Q^{{\boldsymbol{\gamma}}}_n}(f)$ is called 
the step hyperbolic Fourier sum of the function $f$.
According to 
%the introduced above notation,
(\ref{delta_s}), it can be written in the form
$$
S_{Q^{{\boldsymbol{\gamma}}}_n}(f) = 
%\sum\limits_{(s, {\boldsymbol{{\boldsymbol{\gamma}}}})< n}\delta_s(f,\boldsymbol{x}) :=  
\sum\limits_{({\boldsymbol{s}}, {\boldsymbol{\gamma}})< n}\delta_{\boldsymbol{s}}(f).
$$

In connection to the above defined sets of polynomials, we will consider the following approximation characteristics.

 Let $\mathscr{X}$ be a $d$-dimensional, $d\geq 1$, function space defined on $\mathbb{T}^d$ equipped with the norm $\|\cdot\|_{\mathscr{X}}$. Then, for $f \in \mathscr{X}$, by
 $$
 E_{Q^{{\boldsymbol{{\boldsymbol{\gamma}}}}}_n}(f)_{\mathscr{X}} := \inf\limits_{t \in T(Q^{{\boldsymbol{{\boldsymbol{\gamma}}}}}_n)} \| f - t \|_{\mathscr{X}}
 $$
 we define a quantity of the best approximation of the function $f$ by polynomials from the set  $T(Q^{{\boldsymbol{{\boldsymbol{\gamma}}}}}_n)$.

Respectively, for the function class $F\subset \mathscr{X}$,  we set
 \begin{equation}\label{eqno2}
  E_{Q^{{\boldsymbol{{\boldsymbol{\gamma}}}}}_n} (F)_{\mathscr{X}} := \sup\limits_{f \in F}E_{Q^{{\boldsymbol{{\boldsymbol{\gamma}}}}}_n}(f)_{\mathscr{X}}.
\end{equation}
Along with the quantity (\ref{eqno2}), we will investigate the approximation characteristic
 \begin{equation}\label{eqno3}
\mathscr{E}_{Q^{{\boldsymbol{{\boldsymbol{\gamma}}}}}_n}(F)_{\mathscr{X}} := \sup\limits_{f \in F}\, \|f - S_{Q^{{\boldsymbol{{\boldsymbol{\gamma}}}}}_n}(f) \|_{\mathscr{X}}.
\end{equation}

An investigation of the approximation properties of (\ref{eqno2}) and (\ref{eqno3}) on the Sobolev classes $W^r_{p, \alpha}$ and Nikol'skii-Besov classes  $B^{\boldsymbol{r}}_{p,\theta}$  in the case $\mathscr{X} = L_q$ has a reach history. We refer to the monographs  \cite{Dung_Temlyakov_Ullrich2019,Romanyuk2012,Temlyakov1989_121p,Temlyakov1993,Temlyakov2018}.

The aim of this part of the paper is to get estimates for the quantities  (\ref{eqno2}) and (\ref{eqno3})  in the case $F = B^{\boldsymbol{r}}_{p, \theta}$   and $\mathscr{X} = B_{q,1}$.

First, we note the existing connection between the quantities  $E_{Q^{{\boldsymbol{{\boldsymbol{\gamma}}}} }_n}(f)_{B_{q,1}}$  and $\mathscr{E}_{Q^{{\boldsymbol{{\boldsymbol{\gamma}}}} }_n}(f)_{B_{q,1}}$.

On the one hand, immediately from the definitions (\ref{eqno2}) and (\ref{eqno3}), we obtain that for any function
$f \in B_{q,1}$ it holds
 \begin{equation}\label{eqno6}
E_{Q^{{\boldsymbol{{\boldsymbol{\gamma}}}}}_n}(f)_{B_{q,1}} \leq \mathscr{E}_{Q^{{\boldsymbol{{\boldsymbol{\gamma}}}}}_n}(f)_{B_{q,1}}, \qquad 1\leq q\leq \infty.
\end{equation}
Further we show, that in the case $ 1 < q < \infty$ it also holds that
\begin{equation*}%\label{eqno7}
 \mathscr{E}_{Q^{{\boldsymbol{{\boldsymbol{\gamma}}}}}_n}(f)_{B_{q,1}} \ll E_{Q^{{\boldsymbol{{\boldsymbol{\gamma}}}}}_n}(f)_{B_{q,1}}.
\end{equation*}

Let us consider
%$\textbf{S}_{Q^{{\boldsymbol{{\boldsymbol{\gamma}}}}}_n}$ 
%denote 
the 
%Fourier
operator $S_{Q^{{\boldsymbol{{\boldsymbol{\gamma}}}}}_n}$ that puts into a correspondence to the function  $f \in B_{q,1}$, $1 < q < \infty$, its step hyperbolic Fourier sum  $ S_{Q^{{\boldsymbol{{\boldsymbol{\gamma}}}}}_n}(f)$.
%, i.e.,  $\textbf{S}_{Q^{{\boldsymbol{{\boldsymbol{\gamma}}}}}_n } f =  S_{Q^{{\boldsymbol{{\boldsymbol{\gamma}}}}}_n} (f)$.
We 
%Let us
show boundedness of the norm of this operator. 

Indeed, due to (\ref{norm_Bq1}), we have
\begin{align}
\|S_{Q^{{\boldsymbol{{\boldsymbol{\gamma}}}}}_n} \|_{B_{q,1} \rightarrow B_{q,1}} &=
\sup \limits_{\|f\|_{B_{q,1}} \leq 1} \|  S_{Q^{{\boldsymbol{\gamma}}}_n} (f)\|_{B_{q,1}} = \sup \limits_{\|f\|_{B_{q,1}} \leq 1}  \sum\limits_{{\boldsymbol{s}} \in \mathbb{N}^d} \| \delta_{\boldsymbol{s}} (S_{Q^{{\boldsymbol{\gamma}}}_n} (f))\|_q 
\nonumber\\ 
&= \sup \limits_{\|f\|_{B_{q,1}} \leq 1}  \sum\limits_{({\boldsymbol{s}}, {\boldsymbol{\gamma}}) < n} \| \delta_{\boldsymbol{s}} (f)\|_q \leq \sup \limits_{\|f\|_{B_{q,1}} \leq 1}  \sum\limits_{{\boldsymbol{s}} \in \mathbb{N}^d} \| \delta_{\boldsymbol{s}} (f)\|_q 
\nonumber\\ 
& \asymp \sup \limits_{\|f\|_{B_{q,1}} \leq 1} \|f\|_{B_{q,1}} \leq 1.
\label{eqno4}
\end{align}

Further, let $t^{\ast} \in T(Q^{{\boldsymbol{\gamma}}}_n)$ be a polynomial of the best approximation of the function $f \in B_{q,1}$ in the space $B_{q, 1}$. Then, taking into account the fact that  $S_{Q^{{\boldsymbol{\gamma}}}_n}(t^{\ast}) = t^{\ast}$  and using (\ref{eqno4}), we can write
\begin{align}
\mathscr{E}_{Q^{{\boldsymbol{\gamma}}}_n}(f)_{B_{q,1}} &= \| f - S_{Q^{{\boldsymbol{\gamma}}}_n} (f) \|_{B_{q,1}} 
= \| f - t^{\ast} + t^{\ast} - S_{Q^{{\boldsymbol{\gamma}}}_n} (f) \|_{B_{q,1}} 
\nonumber\\ 
&\leq \| f - t^{\ast}\|_{B_{q,1}} + \| S_{Q^{{\boldsymbol{\gamma}}}_n} (f) - t^{\ast} \|_{B_{q,1}} 
 \nonumber\\
&= \| f - t^{\ast}\|_{B_{q,1}}  + \| S_{Q^{{\boldsymbol{\gamma}}}_n} (f - t^{\ast}) \|_{B_{q,1}} 
\nonumber\\ 
&\leq E_{Q^{{\boldsymbol{\gamma}}}_n}(f)_{B_{q,1}} +  \| {S}_{Q^{{\boldsymbol{\gamma}}}_n}\|_{B_{q,1}\rightarrow B_{q,1}
} \|f - t^{\ast}\|_{B_{q,1}}
\nonumber\\ 
&\leq  E_{Q^{{\boldsymbol{\gamma}}}_n}(f)_{B_{q,1}} + C_3(q)\cdot E_{Q^{{\boldsymbol{\gamma}}}_n}(f)_{B_{q,1}} =(1+ C_3(q) )\cdot E_{Q^{{\boldsymbol{\gamma}}}_n}(f)_{B_{q,1}}.
\label{eqno5}
\end{align}

Combining (\ref{eqno6}) and (\ref{eqno5}), we get
 \begin{equation*}%\label{eqno7}
E_{Q^{{\boldsymbol{\gamma}}}_n}(f)_{B_{q,1}} \asymp \mathscr{E}_{Q^{{\boldsymbol{\gamma}}}_n}(f)_{B_{q,1}}, \quad 1<q<\infty.
\end{equation*}

%It is clear that the relations (\ref{eqno7}) also hold  for the quantities  $E_{Q^{{\boldsymbol{\gamma}} '}_n}(f)_{B_{q,1}}$, $\mathscr{E}_{Q^{{\boldsymbol{\gamma}} %'}_n}(f)_{B_{q,1}}$, and, in particular, for   $E_{Q^1_n}(f)_{B_{q,1}}$ and 
%$\mathscr{E}_{Q^1_n}(f)_{B_{q,1}}$.

Let us formulate an auxiliary statement that we will use multiple times in the proofs.

{\bf{Lemma A.}} \cite[
%p. 12
Lemma B, p. 11]{Temlyakov1989_121p} {\it
Let ${\boldsymbol{s}}\in \mathbb{N}^d$, ${\boldsymbol{\gamma}} 
%=(\gamma_1, \dots,\gamma_d)
\in\mathbb{R}^d$, $\gamma_j>0$, $j=1,\dots, d$.
Then
for $\alpha > 0$, the following estimate holds:
$$
\sum\limits_{(s, {\boldsymbol{\gamma}}) \geq l} 2^{- \alpha(s, {\boldsymbol{\gamma}})} \asymp 2^{- \alpha l} l^{d-1}.
$$
If ${\boldsymbol{\gamma}}'
%=({\gamma}'_1, \dots, {\gamma}'_d)
\in\mathbb{R}^d$
is such that ${\gamma}_j={\gamma}_j'=1$ for $j=1,\dots,\nu$ and $1<\gamma_j'<\gamma_j$ for $j=\nu+1, \dots, d$. Then
for $\alpha > 0$, it holds:
$$
\sum\limits_{({\boldsymbol{s}}, {\boldsymbol{\gamma}}') \geq l} 2^{- \alpha({\boldsymbol{s}}, {\boldsymbol{\gamma}})} \asymp 2^{- \alpha l} l^{\nu-1}.
$$
}

\textbf{Theorem A.}
 {\it Let
$$
t(\boldsymbol{x}) = \sum\limits_{|k_j| \leq n_j} c_{\boldsymbol{k}} e^{i(\boldsymbol{k},\boldsymbol{x})},
$$
where 
$\boldsymbol{x}\in\mathbb{R}^d$, $\boldsymbol{k}\in\mathbb{Z}^d$, $\boldsymbol{n}\in\mathbb{N}^d$, $c_{\boldsymbol{k}} \in \mathbb{C}$. 
%$n_j \in \mathbb{N}$, $j = 1, \dots,d$.
Then  for $1 \leq p < q  \leq \infty$   the following inequality holds:
\begin{equation}\label{eqno16}
\|t\|_q \leq 2^d \prod\limits^d_{j=1} n_j^{\frac{1}{p} - \frac{1}{q}} \|t\|_p.
\end{equation}
}

Inequality (\ref{eqno16})  was obtained by S.M. Nikol'skii \cite{Nikol'skii1951}    and is referred to as the inequality for different metrics.

Now we move to 
%formulation and 
a step-by-step proof of the results (see Remark \ref{summung_up_best_approx} for the general statement).
For the details on parameters contained in  the definitions of the function classes (e.g., $\boldsymbol{r}, \boldsymbol{\gamma}, \boldsymbol{\gamma}', \nu$)  were refer to Section \ref{sec1:FCandS}. 

We begin with the exact order estimates for the quantities $\mathscr{E}_{Q^{{\boldsymbol{\gamma}}}_n}(B^{\boldsymbol{r}}_{p,\theta})_{B_{q,1}}$ and $E_{Q^{{\boldsymbol{\gamma}}}_n}(B^{\boldsymbol{r}}_{p,\theta})_{B_{q,1}}$   in the case $1 < p < q < \infty$.

\begin{theorem}\label{Thm2}
Let $d \geq 2$,  $1 < p < q < \infty$, $1 \leq  \theta \leq\ \infty$.   Then for 
$r_1 > 1/p-1/q$
%\frac{1}{p} - \frac{1}{q}$
the following estimates hold
\begin{equation}\label{eqno17}
\mathscr{E}_{Q^{{\boldsymbol{\gamma}}}_n} ( B^{\boldsymbol{r}}_{p,\theta})_{B_{q,1}}    \asymp  E_{Q^{{\boldsymbol{\gamma}}}_n} ( B^{\boldsymbol{r}}_{p,\theta})_{B_{q,1}}   \asymp  2^{ -n( r_1 - \frac{1}{p} + \frac{1}{q})} n^{(\nu - 1)(1 - \frac{1}{\theta})}.
\end{equation}
\end{theorem}

\begin{proof}
To get the upper estimates in (\ref{eqno17}), it is sufficient, due to (\ref{eqno6}), to respectively estimate the quantity  $\mathscr{E}_{Q^{{\boldsymbol{\gamma}}}_n} ( B^{\boldsymbol{r}}_{p,\theta})_{B_{p,1}} $. 

Let $f$ be a function from the class $B^{\boldsymbol{r}}_{p,\theta}$.  Then, from the equivalent representation (\ref{norm_Bq1}) of the norm in the space $B_{q,1}$, we have
\begin{align}
\mathscr{E}_{Q^{{\boldsymbol{\gamma}}}_n} (f)_{B_{q,1}} &= \Big\| f - \sum\limits_{({\boldsymbol{s}}, {\boldsymbol{\gamma}}) < n} \delta_{\boldsymbol{s}}(f) \Big\|_{B_{q,1}} =   \Big\|  \sum\limits_{({\boldsymbol{s}}, {\boldsymbol{\gamma}}) \geq n} \delta_{\boldsymbol{s}}(f) \Big\|_{B_{q,1}}
\nonumber\\ 
&\asymp \sum\limits_{{\boldsymbol{s}} \in \mathbb{N}^d} \Bigg\|\delta_{\boldsymbol{s}} \Bigg(  \mathop{\sum_{{\boldsymbol{s}}' \in \mathbb{N}^d}}\limits_{
({\boldsymbol{s}}', {\boldsymbol{\gamma}}) \geq n} \delta_{{\boldsymbol{s}}'}(f)\Bigg) 
\Bigg\|_q \leq  \sum\limits_{({\boldsymbol{s}}, {\boldsymbol{\gamma}}) \geq n}
\| \delta_{\boldsymbol{s}}(f) \|_q
 =: J_1.
 \label{eqno18}
\end{align}

To further estimate the quantity $J_1$, let us distinguish three cases depending on the value of the parameter $\theta$.

Case $\theta = 1$. Then, taking into account Theorem A and the fact that $({\boldsymbol{s}},{\boldsymbol{1}})\leq ({\boldsymbol{s}},{\boldsymbol{\gamma}})$, we can write
\begin{align*}
J_1 & \leq \sum\limits_{({\boldsymbol{s}},{\boldsymbol{\gamma}} )\geq n}  2^{-({\boldsymbol{s}},{\boldsymbol{\gamma}})r_1}
 2^{ ({\boldsymbol{s}},{\boldsymbol{1}}) (\frac{1}{p} - \frac{1}{q})} 2^{({\boldsymbol{s}},{\boldsymbol{r}})}\|\delta_{\boldsymbol{s}}(f)\|_p
\nonumber\\ 
&\leq \sum\limits_{({\boldsymbol{s}},{\boldsymbol{\gamma}} )\geq n} 2^{-({\boldsymbol{s}},{\boldsymbol{\gamma}})(r_1 - \frac{1}{p} + \frac{1}{q})} 2^{({\boldsymbol{s}},{\boldsymbol{r}})}  \|\delta_{\boldsymbol{s}}(f)\|_p
\nonumber\\ 
&
\leq 2^{-n (r_1 - \frac{1}{p} + \frac{1}{q}) } \sum\limits_{(\boldsymbol{s},{\boldsymbol{\gamma}} )\geq n} 2^{(\boldsymbol{s},\boldsymbol{r})}  \|\delta_{\boldsymbol{s}}(f)\|_p 
\nonumber\\ 
&
\leq
 2^{-n (r_1 - \frac{1}{p} + \frac{1}{q}) } \|f\|_{B^{\boldsymbol{r}}_{p,1}}
\leq
% 2^{-n (r_1 - \frac{1}{p} + \frac{1}{q}) } \|f\|_{B^{\boldsymbol{r}}_{p,1}} \leq
 2^{-n (r_1 - \frac{1}{p} + \frac{1}{q}) } .
% \label{eqno22}
\end{align*}

 Case $ 1 < \theta < \infty$. Using the H\"{o}lder's inequality with the power $\theta$ to $J_1$, we get
 \begin{align}
 J_1 &\leq \sum\limits_{({\boldsymbol{s}},{\boldsymbol{\gamma}} )\geq n} 2^{-(({\boldsymbol{s}},{\boldsymbol{r}}) - ({\boldsymbol{s}},{\boldsymbol{1}}) (\frac{1}{p} - \frac{1}{q}))} 2^{({\boldsymbol{s}},{\boldsymbol{r}})} \|\delta_{\boldsymbol{s}}(f)\|_p
\nonumber \\ 
&
 \leq \Bigg(   \sum\limits_{({\boldsymbol{s}},{\boldsymbol{\gamma}} )\geq n} 2^{( {\boldsymbol{s}},{\boldsymbol{r}})\theta} \|\delta_{\boldsymbol{s}}(f)\|^{\theta}_p \Bigg)^{\frac{1}{\theta}} \Bigg( \sum\limits_{({\boldsymbol{s}},{\boldsymbol{\gamma}} )\geq n} 2^{-(({\boldsymbol{s}},{\boldsymbol{r}}) - ({\boldsymbol{s}},{\boldsymbol{1}}) (\frac{1}{p} - \frac{1}{q}))\theta' } \Bigg)^{\frac{1}{\theta'}}
\nonumber \\ 
&
 \ll    \|f\|_{B^{\boldsymbol{r}}_{p,\theta}} \, \Bigg( \sum\limits_{({\boldsymbol{s}},{\boldsymbol{\gamma}} )\geq n} 2^{-({\boldsymbol{s}},{\boldsymbol{r}} - {\boldsymbol{1}}(\frac{1}{p} + \frac{1}{q}))\theta' } \Bigg)^{\frac{1}{\theta'}}
 \leq \Bigg( \sum\limits_{({\boldsymbol{s}},{\boldsymbol{\gamma}} )\geq n} 2^{-({\boldsymbol{s}},\widetilde{{\boldsymbol{\gamma}}})(r_1 - \frac{1}{p} + \frac{1}{q})\theta' } \Bigg)^{\frac{1}{\theta'}},
 \label{eqno19}
\end{align}
where 
%{\boldsymbol{r}} - 1/p + 1/q$ is a vector with the coordinates
% $r_j - 1/p + 1/q$, $j=1, \dots, d$, and, respectively, 
$\widetilde{{\boldsymbol{\gamma}}}
 %= (\widetilde{\gamma_1}, \ldots, \widetilde{{\gamma}_d})
 \in \mathbb{R}^d$ is such that
 $\widetilde{\gamma}_j = (r_j - 1/p + 1/q)/(r_1 - 1/p + 1/q)$, $j=1, \dots, d$.
  
   One can easily check that
   $\widetilde{\gamma}_j = {\gamma}_j=1$  for $j =1, \dots, \nu$  and $1 < {\gamma}_j < \widetilde{{\gamma}}_j$ for $j = \nu +1, \dots, d$.
From this, using Lemma A for the last sum in (\ref{eqno19}), we obtain
\begin{equation*}%\label{eqno20}
J_1 \ll 2^{ -n( r_1 - \frac{1}{p} + \frac{1}{q})} n^{(\nu - 1)(1 - \frac{1}{\theta})}.
\end{equation*}

Case $\theta = \infty$.  
Taking into account, that for    $f \in B^{\boldsymbol{r}}_{p,\infty}$, $1 < p < \infty$, it holds  (see (\ref{eq1_infty}))
\begin{equation}\label{norm_delta_infty}
  \|\delta_{\boldsymbol{s}}(f)\|_p \ll 2^{-({\boldsymbol{s}},{\boldsymbol{r}})}, \quad {\boldsymbol{s}} \in \mathbb{N}^d,  
\end{equation}
using Lemma A, we have that
\begin{align}
J_1 & \ll \sum\limits_{({\boldsymbol{s}},{\boldsymbol{\gamma}} )\geq n} 2^{-(({\boldsymbol{s}},{\boldsymbol{r}}) - ({\boldsymbol{s}},{\boldsymbol{1}}) (\frac{1}{p} - \frac{1}{q}))} =
\sum\limits_{({\boldsymbol{s}},{\boldsymbol{\gamma}} )\geq n} 2^{-({\boldsymbol{s}},{\boldsymbol{r}} - \boldsymbol{1}(\frac{1}{p} + \frac{1}{q}))}
\nonumber\\ 
&
= \sum\limits_{({\boldsymbol{s}},{\boldsymbol{\gamma}} )\geq n} 2^{-({\boldsymbol{s}},\widetilde{{\boldsymbol{\gamma}}})(r_1 - \frac{1}{p} + \frac{1}{q})} \ll  2^{-n (r_1 - \frac{1}{p} + \frac{1}{q}) } n^{\nu -1}.
\label{eqno21}
\end{align}

Combining (\ref{eqno18}) -- (\ref{eqno21}), we obtain the upper estimate for the quantity
 $\mathscr{E}_{Q^{{\boldsymbol{\gamma}}}_n} ( B^{\boldsymbol{r}}_{p,\theta})_{B_{q,1}} $, and hence for $E_{Q^{{\boldsymbol{\gamma}}}_n} ( B^{\boldsymbol{r}}_{p,\theta})_{B_{q,1}} $.

It is sufficient to get the lower estimate in (\ref{eqno17}) for the case
$\nu = d$. Below we construct a function that realizes the obtained upper order estimate.

Let
$$
g(\boldsymbol{x}) = C_4 \, 2^{-n(r_1 + 1 - \frac{1}{p})} n^{-{\frac{d-1}{\theta}}} d_n(\boldsymbol{x}), \quad C_4 >0,
$$
where
$$
d_n(\boldsymbol{x}) = \sum\limits_{({\boldsymbol{s}},{\boldsymbol{1}})= n} \, \sum\limits_{\boldsymbol{k} \in \rho({\boldsymbol{s}})} e^{i(\boldsymbol{k},\boldsymbol{x})}.
$$

Let us first show
 that  for certain choice of the constant  $C_4$,  $g \in B^{\boldsymbol{r}}_{p,\theta}$, ${\boldsymbol{r}}=(r_1, \dots, r_1)\in \mathbb{R}^d$, $r_1>0$, $1<p<\infty$,
$1 \leq \theta  \leq \infty$.
%it belongs to the class $B^{\boldsymbol{r}}_{p,\theta}$.

Let first $1\leq \theta <\infty$. Then we can write
\begin{align}
\|g\|_{B^{\boldsymbol{r}}_{p,\theta}} & \asymp \Bigg(    \sum\limits_{({\boldsymbol{s}}, {\boldsymbol{1}} ) = n} 2^{( {\boldsymbol{s}},{\boldsymbol{r}})\theta} \|\delta_{\boldsymbol{s}}(g)\|^{\theta}_p \Bigg)^{\frac{1}{\theta}}  \nonumber\\
&
\asymp 2^{-n(r_1 + 1 - \frac{1}{p})} n^{-{\frac{d-1}{\theta}}} \Bigg(    \sum\limits_{({\boldsymbol{s}},\boldsymbol{1} ) = n} 2^{( {\boldsymbol{s}},{\boldsymbol{r}})\theta} \|\delta_{\boldsymbol{s}}(d_n)\|^{\theta}_p \Bigg)^{\frac{1}{\theta}}
\nonumber\\ 
&
=  2^{-n(1 - \frac{1}{p})} n^{-{\frac{d-1}{\theta}}} \Bigg(    \sum\limits_{({\boldsymbol{s}},{\boldsymbol{1}} ) = n}  \|\delta_{\boldsymbol{s}}(d_n)\|^{\theta}_p \Bigg)^{\frac{1}{\theta}}.
%=: J_3.
 \label{eqno23}
\end{align}

Then we use the known relation
$$
\Big\| \sum\limits^m_{k = -m}  e^{i k x} \Big\|_p \asymp m^{1 - \frac{1}{p}}, \quad 1 < p < \infty, \ x \in \mathbb{R}
$$
(see, e.g., \cite[Ch. 1, Paragraph 1]{Temlyakov1993}), and get
\begin{equation}\label{eqno24}
\|\delta_{\boldsymbol{s}}(d_n)\|_p \asymp 2^{({\boldsymbol{s}}, {\boldsymbol{1}}) (1 - \frac{1}{p})}.
\end{equation}
Substituting (\ref{eqno24}) into (\ref{eqno23}), we continue the estimation of 
%$J_3$
$\|g\|_{B^{\boldsymbol{r}}_{p,\theta}}$
as follows
\begin{align*}
\|g\|_{B^{\boldsymbol{r}}_{p,\theta}} & \asymp 2^{-n(1 - \frac{1}{p})} n^{-{\frac{d-1}{\theta}}} \Bigg(    \sum\limits_{({\boldsymbol{s}},{\boldsymbol{1}} ) = n} 2^{({\boldsymbol{s}},{\boldsymbol{1}})(1 - \frac{1}{p})\theta} \Bigg)^{\frac{1}{\theta}}
\nonumber\\ 
&
= n^{-{\frac{d-1}{\theta}}} \Bigg(    \sum\limits_{({\boldsymbol{s}},{\boldsymbol{1}} ) = n} 1 \Bigg)^{\frac{1}{\theta}} \asymp   n^{-{\frac{d-1}{\theta}}}  n^{\frac{d-1}{\theta}} = 1.
 %\label{eqno25}
\end{align*}
%Hence, form (\ref{eqno23}) and (\ref{eqno25}) we obtain that  $g_2 \in B^{\boldsymbol{r}}_{p,\theta}$, $1 \leq \theta < \infty$.

Let further $\theta = \infty$.  Then the function  $g$  takes the form
$$
g(\boldsymbol{x}) = C_4 \, 2^{-n(r_1 + 1 - \frac{1}{p})} d_n(\boldsymbol{x})
$$
and therefore, in view of (\ref{eq1}) and (\ref{eqno24}), we get
$$
\|g\|_{B^{\boldsymbol{r}}_{p,\infty}} \asymp \sup\limits_{{\boldsymbol{s}} \in \mathbb{N}^d} 2^{({\boldsymbol{s}},{\boldsymbol{r}})} \|\delta_{\boldsymbol{s}}(g)\|_p \asymp \sup\limits_{{\boldsymbol{s}}: ({\boldsymbol{s}},{\boldsymbol{1}}) = n}\, 2^{({\boldsymbol{s}},{\boldsymbol{r}})} 2^{-n(r_1 + 1 - \frac{1}{p})}\|\delta_{\boldsymbol{s}}(d_n)\|_p \asymp 1.
$$
We make a conclusion that for an appropriate choice of the constant  $C_4>0$, the function $g \in B^{\boldsymbol{r}}_{p,\theta}$.

Then, taking into account that $S_{Q^{\boldsymbol{1}}_n}(g) = 0$ and the relation (\ref{eqno24}), we can write
\begin{align*}
E_{Q^{\boldsymbol{1}}_n} (g)_{B_{q,1}} & \asymp  \mathscr{E}_{Q^{\boldsymbol{1}}_n} (g)_{B_{q,1}} = \|g\|_{B_{q,1}}
\asymp  2^{-n(r_1 + 1  - \frac{1}{p})} n^{-{\frac{d-1}{\theta}}}  \sum\limits_{({\boldsymbol{s}},{\boldsymbol{1}} ) = n} \|\delta_ {\boldsymbol{s}}(d_n)\|_q 
\\ 
&
\asymp  2^{-n(r_1 + 1  - \frac{1}{p})} n^{-{\frac{d-1}{\theta}}}  \sum\limits_{({\boldsymbol{s}},{\boldsymbol{1}} ) = n} 2^{({\boldsymbol{s}},{\boldsymbol{1}})(1 - \frac{1}{q})}
%\\ &
\asymp 
%2^{-n(r_1 + 1  - \frac{1}{p})} n^{-{\frac{d-1}{\theta}}} 2^{n(1 %- \frac{1}{q})} n^{d-1} =
2^{-n(r_1  - \frac{1}{p} + \frac{1}{q})} n^{(d-1)(1 - \frac{1}{\theta})}.
\end{align*}

Theorem \ref{Thm2} is proved
\end{proof}

In what follows, we separately consider the case $p=q$, since in this situation the orders of the corresponding approximation quantities for $Q^{{\boldsymbol{\gamma}}}_n$ and $Q^{{\boldsymbol{\gamma}}'}_n$ with 
%${\boldsymbol{\gamma}}'  = ({\gamma}'_1, \ldots, \gamma'_d)$,  where 
$\gamma'_j = \gamma_j=1$  if  $j = 1,\dots, \nu$  and $1 < \gamma'_j < \gamma_j$  if 
  %$j = \overline{\nu +1, d }$.
$j = \nu +1, \dots, d$, differ.

\begin{theorem}\label{Thm1}
  Let $d \geq 2$, $1 < p < \infty$, $1 \leq \theta \leq \infty$. Then for $r_1 > 0$ it holds
\begin{equation*}%\label{eqno8}
\mathscr{E}_{Q^{{\boldsymbol{\gamma}}'}_n} ( B^{\boldsymbol{r}}_{p,\theta})_{B_{p,1}}    \asymp  E_{Q^{{\boldsymbol{\gamma}}'}_n} ( B^{\boldsymbol{r}}_{p,\theta})_{B_{p,1}}   \asymp  2^{ -n  r_1} n^{(\nu - 1)(1 - \frac{1}{\theta})}.
\end{equation*}
\end{theorem}

\begin{proof} We will argue similarly to that in Theorem~\ref{Thm2}.
For $f\in B^{\boldsymbol{r}}_{p,\theta}$, we obtain
\begin{align*}
\mathscr{E}_{Q^{{\boldsymbol{\gamma}}'}_n} (f)_{B_{p,1}}  \leq  
\sum\limits_{(s, {\boldsymbol{\gamma}}') \geq n}\| \delta_s(f) \|_p
 =: J_2.
 %\label{eqno9}
\end{align*}

Further, as above, we distinguish three cases.

For the case $\theta = 1$, we note that 
$2^{-({\boldsymbol{s}},{\boldsymbol{\gamma}})r_1}\leq 2^{-({\boldsymbol{s}},{\boldsymbol{\gamma}}')r_1}$.
%we can write
%\begin{equation}\label{eqno10.1}
%J_2  =   \sum\limits_{(s, {\boldsymbol{\gamma}}') \geq n} 2^{-(s,{\boldsymbol{\gamma}})r_1}
%2^{(s,r)}  \|\delta_s(f) \|_p  \leq
%2^{-n r_1} \|f\|_{B^{\boldsymbol{r}}_{p,1}} \leq 2^{-n r_1}.
%\end{equation}

In the case $1 
%\leq 
<\theta < \infty$, by  Lemma A, in view of ${\boldsymbol{\gamma}} = {\boldsymbol{r}}/r_1$, we obtain
\begin{align*}
J_2 & \ll \|f\|_{B^{\boldsymbol{r}}_{p,\theta}}  \Bigg(  \sum\limits_{({\boldsymbol{s}}, {\boldsymbol{\gamma}}') \geq n} 2^{-({\boldsymbol{s}},{\boldsymbol{r}})\theta'}  \Bigg)^{\frac{1}{\theta'}}
\leq \Bigg(  \sum\limits_{({\boldsymbol{s}}, {\boldsymbol{\gamma}}') \geq n} 2^{-({\boldsymbol{s}},{\boldsymbol{\gamma}}) r_1\, \theta'}  \Bigg)^{\frac{1}{\theta'}} 
\nonumber\\ 
&\ll 2^{-n r_1} n^{(\nu - 1)(1 - \frac{1}{\theta})}.
% \label{eqno10}
\end{align*}

Case $\theta = \infty$. In view of (\ref{norm_delta_infty}),
%the fact, that for $f  \in B^{\boldsymbol{r}}_{p,\infty}$ the relation
%$\|\delta_s(f)\|_p \ll 2^{-(s,r)}$, $s \in \mathbb{N}^d$,
%holds,
 Lemma A yields
\begin{equation*}%\label{eqno11}
J_2 \ll \sum\limits_{({\boldsymbol{s}}, {\boldsymbol{\gamma}}') \geq n} 2^{-({\boldsymbol{s}}, {\boldsymbol{r}})} \ll 2^{-n r_1} n^{\nu - 1}.
\end{equation*}

For the respective lower estimate, it is sufficient to consider the function $g\in B^{\boldsymbol{r}}_{p,\theta}$ from the proof of Theorem~\ref{Thm2}.

Theorem \ref{Thm1} is proved.
\end{proof}

\begin{remark}
%one can check that 
Under the conditions of Theorem \ref{Thm1}, for the quantities $\mathscr{E}_{Q^{{\boldsymbol{\gamma}}}_n}(B^{\boldsymbol{r}}_{p,\theta})_{B_{p,1}}$    and $E_{Q^{{\boldsymbol{\gamma}}}_n}(B^{\boldsymbol{r}}_{p,\theta})_{B_{p,1}}$ it holds
\begin{equation}\label{eqno15}
\mathscr{E}_{Q^{{\boldsymbol{\gamma}}}_n}(B^{\boldsymbol{r}}_{p,\theta})_{B_{p,1}} \asymp E_{Q^{{\boldsymbol{\gamma}}}_n}(B^{\boldsymbol{r}}_{p,\theta})_{B_{p,1}} \asymp  2^{-n r_1} n^{(d - 1)(1- \frac{1}{\theta})}.
\end{equation}
This is a corollary from (see Lemma A)
$$
\sum\limits_{({\boldsymbol{s}},{\boldsymbol{\gamma}} )\geq n} 2^{-({\boldsymbol{s}},{\boldsymbol{r}})} =
\sum\limits_{({\boldsymbol{s}},{\boldsymbol{\gamma}} )\geq n} 2^{-({\boldsymbol{s}},{\boldsymbol{\gamma}})r_1} \asymp
2^{-n r_1} n^{d - 1}.
$$
\end{remark}

To comment on the obtained results in Theorems \ref{Thm2} and \ref{Thm1},
first, let us note that the corresponding statements, where the error norm is measured  in  $L_q$,  are known. Let us formulate them for convenience.

\textbf{Theorem B.}
 {  \it
Let $d \geq 2$.
%, $1 < p = q < \infty$, $1 \leq  \theta \leq \infty$, 
Then for $r_1 > 0$ it holds
%\begin{equation*}
%\mathscr{E}_{Q^{{\boldsymbol{\gamma}}'}_n} ( B^{\boldsymbol{r}}_{p,\theta})_q    \asymp  E_{Q^{{\boldsymbol{\gamma}}'}_n} %( B^{\boldsymbol{r}}_{p,\theta})_q   \asymp  2^{ -n( r_1 - \frac{1}{p} + \frac{1}{q})} %n^{(\nu - 1)(\frac{1}{q^*} - \frac{1}{\theta})_{+}},
% \end{equation*}
%}
%Let $d \geq 2$, $r_1 > 0$.  Then the folllwing estimates hold:
$$
\mathscr{E}_{Q^{{\boldsymbol{\gamma}}'}_n} ( B^{\boldsymbol{r}}_{p,\theta})_p   \asymp  E_{Q^{{\boldsymbol{\gamma}}'}_n} ( B^{\boldsymbol{r}}_{p,\theta})_p
 \asymp \left\{
 \begin{array}{ll}
          2^{ -n  r_1} n^{(\nu - 1)(\frac{1}{p} - \frac{1}{\theta})}  ,  & 1 < p \leq 2, \
 p < \theta \leq \infty,  \\
          2^{ -n  r_1} n^{(\nu - 1)(\frac{1}{2} - \frac{1}{\theta})}  ,  & 2 < p < \infty, \
 2 < \theta \leq \infty.
          \end{array}
        \right.
$$

In the cases $1 < p \leq 2$  and $1 \leq \theta \leq p$, or  $2 < p < \infty$  and $1 \leq \theta \leq 2$ it holds
$$
\mathscr{E}_{Q^{{\boldsymbol{\gamma}}}_n} ( B^{\boldsymbol{r}}_{p,\theta})_p   \asymp  E_{Q^{{\boldsymbol{\gamma}}}_n} ( B^{\boldsymbol{r}}_{p,\theta})_p   \asymp 2^{-n r_1}.
$$}

In the case  $\theta = \infty$, 
% the formulated above estimates were obtained 
Theorem B was proved
in \cite{Bugrov1964} for $p=2$  and in \cite{Nikolskaja1974}  for $1 < p < \infty$. In the case $1 \leq \theta < \infty$, 
%Theorem B was proved
the corresponding estimates were obtained
in  \cite{Romanyuk1991}.

\textbf{Theorem C.}
 {  \it Let $d \geq 2$,  $1 < p < q < \infty$, $1 \leq  \theta \leq \infty$.   Then for $r_1 > 
 %\frac{1}{p} - \frac{1}{q}
 1/p-1/q$   the following relations hold
\begin{equation*}%\label{eqno26}
\mathscr{E}_{Q^{{\boldsymbol{\gamma}}}_n} ( B^{\boldsymbol{r}}_{p,\theta})_q    \asymp  E_{Q^{{\boldsymbol{\gamma}}}_n} ( B^{\boldsymbol{r}}_{p,\theta})_q   \asymp  2^{ -n( r_1 - \frac{1}{p} + \frac{1}{q})} n^{(\nu - 1)(\frac{1}{q} - \frac{1}{\theta})_{+}},
 \end{equation*}
  where
$a_{+} = \max \{a, 0 \}$.
}

In the case $\theta = \infty$, 
%the estimates (\ref{eqno26})
Theorem C 
%were
was proved 
%obtained 
in  \cite{Temlyakov1983} for  $1 < p < q \leq 2$  and  in \cite{Temlyakov1989_121p} for $1 < p < q < \infty$, $q > 2$, respectively.
 In the case $1 \leq \theta < \infty$, 
 %Theorem A was proved 
 the corresponding estimates were obtained in  \cite{Romanyuk1991}.

\begin{remark}
Comparing the results of Theorems \ref{Thm2} and  \ref{Thm1} (see also (\ref{eqno15})) and Theorems B and C, we make the conclusion:
only in the cases $\theta = 1$  or $\nu = 1$ ($d=1$ respectively) the corresponding approximation characteristics of the classes $B^{\boldsymbol{r}}_{p,\theta}$ in the spaces  $B_{q,1}$   and  $L_q$, $1 < p\leq q < \infty$, coincide in order and, if additionally $p=q$, they do not depend on the parameter $p$.  %For all other values of the parameters $r$, $p$  and $\theta$,
In all other cases, their orders differ.
 \end{remark}

Returning to Theorem \ref{Thm1}, we note that it does not cover the limiting cases  $p \in \{ 1, \infty\}$.

%
%since the equivalent representations (\ref{eq1}) and %%(\ref{eq1_infty}) for the norm of $f\in B^{\boldsymbol{r}}_{p,\theta}$ in terms of the blocks $\delta_s(f)$ hold only for $1<p<\infty$.

The next statement concerns these limiting cases, but only for the quantity 
$E_{Q^{{\boldsymbol{\gamma}}'}_n} ( B^{\boldsymbol{r}}_{p,\theta})_{B_{p,1}}$. 

\begin{theorem}\label{Thm3}
Let $d \geq 2$,  $p \in \{1, \infty\}$, $1 \leq  \theta \leq \infty$. Then for $r_1 > 0$ it holds
\begin{equation}\label{eqno27}
 E_{Q^{{\boldsymbol{\gamma}}'}_n} ( B^{\boldsymbol{r}}_{p,\theta})_{B_{p,1}}   \asymp  2^{ -n  r_1} n^{(\nu - 1)(1 - \frac{1}{\theta})}.
\end{equation}
\end{theorem}
\begin{proof}
We begin with the upper estimate.
%Let first $f \in B^{\boldsymbol{r}}_{p,\theta}$, $1 \leq \theta < \infty$.
As an approximation aggregate for a function $f\in B^{\boldsymbol{r}}_{p,\theta}$, we take the polynomial
\begin{equation}\label{eqno27.1}
t_n :=t_n(f,\boldsymbol{x}) =  \sum\limits_{( {\boldsymbol{s}}, {\boldsymbol{\gamma}}' ) < n - ({\boldsymbol{\gamma}}' , {\boldsymbol{1}})}  A_{\boldsymbol{s}}(f,\boldsymbol{x}), \quad \boldsymbol{x} \in \mathbb{R}^d.
\end{equation}
%where  
%${\boldsymbol{\gamma}}'(d)= {\boldsymbol{\gamma}}'_1 + \dots + {\boldsymbol{\gamma}}'_d$. 
One can easily check that  $t_n \in T(Q^{{\boldsymbol{\gamma}}'}_n)$.

Then, by the definition (\ref{norm_Bq1_via_As}) of the norm in the space $B_{p,1}$, $p \in \{ 1,\infty \}$,  we can write
\begin{align}
E_{Q^{{\boldsymbol{\gamma}}'}_n} (f)_{B_{p,1}} & \leq  \|f - t_n\|_{B_{p,1}} =
\Big\| \sum\limits_{({\boldsymbol{s}}, {\boldsymbol{\gamma}}') \geq n - ({\boldsymbol{\gamma}}' , {\boldsymbol{1}})} \, A_{\boldsymbol{s}}(f)\Big\|_{B_{p,1}}
\nonumber\\ 
&  =\sum\limits_{ {\boldsymbol{s}} \in \mathbb{N}^d} \,
\Bigg\| A_{\boldsymbol{s}}  \ast \, \mathop{\sum_{{\boldsymbol{s}}' \in \mathbb{N}^d}} \limits_{({\boldsymbol{s}}', {\boldsymbol{\gamma}}') \geq n - ({\boldsymbol{\gamma}}' , {\boldsymbol{1}})} \, A_{{\boldsymbol{s}}'}(f)\Bigg\|_p 
\nonumber\\
&\leq
\sum\limits_{({\boldsymbol{s}}, {\boldsymbol{\gamma}}') \geq n - 2 ({\boldsymbol{\gamma}}' , {\boldsymbol{1}})} \, \Big\|A_{\boldsymbol{s}}  \ast  \sum\limits_{\| {\boldsymbol{s}} - {\boldsymbol{s}}'\|_{\infty} \leq 1} \, A_{{\boldsymbol{s}}'}(f) \Big\|_p
\nonumber\\ 
&
\leq \sum\limits_{({\boldsymbol{s}}, {\boldsymbol{\gamma}}') \geq n - 2 ({\boldsymbol{\gamma}}' , {\boldsymbol{1}})} \, \|A_{\boldsymbol{s}}\|_1 \,
\Big\| \sum\limits_{\| {\boldsymbol{s}} - {\boldsymbol{s}}'\|_{\infty} \leq 1} \, A_{{\boldsymbol{s}}'}(f) \Big\|_p . 
%= J_4.
\label{eqno28}
\end{align}

Then we use the fact that, in view of the relation
 $\|V_{2^s}\|_1 \leq C_5$, $C_5 > 0$  (see, e.g., \cite[Ch. 1, Paragraph 1]{Temlyakov1993}), it holds
\begin{equation}\label{norm_As_1}
   \|A_{\boldsymbol{s}}\|_1 =
   \Big\|\prod_{j=1}^d (V_{2^{s_j}} - V_{2^{s_j-1}})\Big\|_1 \leq  
   \Big\|\prod_{j=1}^d V_{2^{s_j}} \Big\|_1 +   \Big\|\prod_{j=1}^d V_{2^{s_j-1}} \Big\|_1 \leq C_6, \quad C_6 >0. 
\end{equation}

Hence, from (\ref{eqno28})  we obtain
\begin{align}
%J_4 
E_{Q^{{\boldsymbol{\gamma}}'}_n} (f)_{B_{p,1}}
&\ll  \sum\limits_{({\boldsymbol{s}}, {\boldsymbol{\gamma}}') \geq n - 2 ({\boldsymbol{\gamma}}' , {\boldsymbol{1}})} \,  \sum\limits_{\| {\boldsymbol{s}} - {\boldsymbol{s}}'\|_{\infty} \leq 1} \, \|A_{{\boldsymbol{s}}'}(f) \|_p
\ll   \sum\limits_{({\boldsymbol{s}}, {\boldsymbol{\gamma}}') \geq n - 3 ({\boldsymbol{\gamma}}' , {\boldsymbol{1}})} \|A_{\boldsymbol{s}}(f)\|_p
\nonumber\\ 
&
  =  
 \sum\limits_{({\boldsymbol{s}}, {\boldsymbol{\gamma}}') \geq n - 3 ({\boldsymbol{\gamma}}' , {\boldsymbol{1}})} 2^{({\boldsymbol{s}}, {\boldsymbol{r}})}\, \|A_{\boldsymbol{s}}(f)\|_p\,  2^{-({\boldsymbol{s}},{\boldsymbol{r}})}.
 %= J_5.
\label{eqno29}
\end{align}

Let first $1 \leq \theta < \infty$. Using the H\"{o}lder's inequality with the power $\theta$
%to $J_5$ 
(with an appropriate modification 
%of this inequality 
in the case $\theta = 1$), and then in view of
(\ref{norm_via_A_s_notinfty}),  ${\boldsymbol{\gamma}} = {\boldsymbol{r}}/r_1$ and Lemma A,  from (\ref{eqno29}) we get
\begin{align*}
%J_5
E_{Q^{{\boldsymbol{\gamma}}'}_n} (f)_{B_{p,1}} & \ll  \Bigg( \sum\limits_{({\boldsymbol{s}}, {\boldsymbol{\gamma}}') \geq n - 3 ({\boldsymbol{\gamma}}' , {\boldsymbol{1}})} 2^{({\boldsymbol{s}},{\boldsymbol{r}})\theta} \|A_{\boldsymbol{s}}(f)\|^{\theta}_p \Bigg)^{\frac{1}{\theta}} \Bigg( \sum\limits_{({\boldsymbol{s}}, {\boldsymbol{\gamma}}') \geq n - 3 ({\boldsymbol{\gamma}}' , {\boldsymbol{1}})} 2^{-({\boldsymbol{s}},{\boldsymbol{r}})\theta'} \Bigg)^{\frac{1}{\theta'}}
\nonumber\\ 
&
  \ll   \|f\|_{B^{\boldsymbol{r}}_{p,\theta}} \Bigg(  \sum\limits_{({\boldsymbol{s}}, {\boldsymbol{\gamma}}') \geq n - 3 ({\boldsymbol{\gamma}}' , {\boldsymbol{1}})} 2^{-({\boldsymbol{s}},{\boldsymbol{\gamma}}) r_1 \theta'} \Bigg)^{\frac{1}{\theta'}} 
  %\leq \Bigg(  \sum\limits_{(s, {\boldsymbol{\gamma}}') \geq n - 3 ({\boldsymbol{\gamma}}' , 1)} 2^{-(s,{\boldsymbol{\gamma}}) r_1 \theta'} \Bigg)^{\frac{1}{\theta'}}
%\nonumber  \\ 
%&
  \asymp  2^{ -n  r_1} n^{(\nu - 1)(1 - \frac{1}{\theta})}.
 % \label{eqno30}
\end{align*}

%Therefore, the estimates (\ref{eqno28}) -- (\ref{eqno30}) yield the upper estimate of the quantity %$E_{Q^{{\boldsymbol{\gamma}}'}_n} ( B^{\boldsymbol{r}}_{p,\theta})_{B_{p,1}}$  in the case $\theta \in [\,1, \infty)$.

If $\theta = \infty$, then, in view of the relation $\|A_{\boldsymbol{s}}(f)\|_p \ll 2^{-({\boldsymbol{s}},{\boldsymbol{r}})}$ 
%that holds for the function   $f \in B^{\boldsymbol{r}}_{p,\infty}$, $p \in \{ 1, \infty\}$, 
(see (\ref{norm_via_A_s_infty})), from (\ref{eqno29}) for $p \in \{1, \infty\}$
we get  
\begin{align*}
 E_{Q^{{\boldsymbol{\gamma}}'}_n} (f)_{B_{p,1}}
 %  &
 %\leq    \|f - t_n \|_{B_{p,1}}
 \ll \sum\limits_{({\boldsymbol{s}}, {\boldsymbol{\gamma}}') \geq n - 3 ({\boldsymbol{\gamma}}' , 1)} \|A_{\boldsymbol{s}}(f)\|_p
% \\
 %& 
 \ll  
 %\sum\limits_{(s, {\boldsymbol{\gamma}}') \geq n - 3 ({\boldsymbol{\gamma}}' , 1)}2^{-(s, r)} = 
 \sum\limits_{({\boldsymbol{s}}, {\boldsymbol{\gamma}}') \geq n - 3 ({\boldsymbol{\gamma}}' , {\boldsymbol{1}})}2^{-({\boldsymbol{s}}, {\boldsymbol{\gamma}})r_1}  \asymp 2^{-n r_1}n^{\nu - 1}.
\end{align*}

The upper estimate is now proved.

Concerning the lower estimate in (\ref{eqno27}), we note the following.

%In the case $p=1$, 
The estimate of the quantity $E_{Q^{{\boldsymbol{\gamma}}'}_n} ( B^{\boldsymbol{r}}_{1,\theta})_{B_{1,1}}$ is a corollary from 
Theorem~D (see below) and the relation  $\|\cdot\|_{B_{1,1}} \gg \|\cdot\|_1$. 
%For $p = \infty$, 
The lower estimate of   $E_{Q^{{\boldsymbol{\gamma}}'}_n} ( B^{\boldsymbol{r}}_{\infty,\theta})_{B_{\infty,1}}$ follows from the corresponding estimate of the Kolmogorov widths
(see Theorem \ref{Thm8} and Remark \ref{Rem3}),  that will be obtained in 
%the next part of the paper 
Section~\ref{sec3:widths_en}.

%in view of the relation 
%$$
%E_{Q^{{\boldsymbol{\gamma}}'}_n} ( B^{\boldsymbol{r}}_{\infty,\theta})_{B_{\infty,1}} \gg d_M(B^{\boldsymbol{r}}_{\infty, \theta}, B_{\infty,1}), 
%\quad M \asymp 2^n n^{\nu - 1}.
%$$

Theorem \ref{Thm3} is proved.
\end{proof}
\begin{remark}
%one can check that 
One can show that under the conditions of Theorem \ref{Thm3} it holds
\begin{equation}\label{eqno27.gamma}
 E_{Q^{{\boldsymbol{\gamma}}}_n} ( B^{\boldsymbol{r}}_{p,\theta})_{B_{p,1}}   \asymp  2^{ -n  r_1} n^{(d - 1)(1 - \frac{1}{\theta})}.
\end{equation}
\end{remark}
%In what follows, we comment the obtained result in Theorem \ref{Thm3}.

%First, we note that 

To compare the estimate (\ref{eqno27}) with the corresponding results in the space  $L_p$, $p \in \{ 1, \infty \}$,
let us formulate the known statements.
%for this space.

 \textbf{Theorem D. }   {  \it Let $d \geq 2$, $1 \leq  \theta \leq \infty$.   Then for $r_1 > 0$ it holds
\begin{equation}\label{eqno31}
 E_{Q^{{\boldsymbol{\gamma}}'}_n} ( B^{\boldsymbol{r}}_{1,\theta})_1   \asymp  2^{ -n  r_1} n^{(\nu - 1)(1 - \frac{1}{\theta})}.
\end{equation}
}
The estimate (\ref{eqno31})  in the case $\theta = \infty$ was obtained in \cite{Temlyakov1980}, 
and in the case $1 \leq \theta < \infty$ in \cite{Romanyuk2004}.

\textbf{Theorem E.}   {
\it Let $d = 2$, $1 \leq  \theta \leq \infty$.   Then for $\boldsymbol{r}=(r_1, r_1)$, $r_1 > 0$, it holds
\begin{equation}\label{eqno32}
 E_{Q^{{\boldsymbol{\gamma}}}_n} ( B^{\boldsymbol{r}}_{\infty,\theta})_\infty    \asymp  2^{ -n  r_1} n^{1 - \frac{1}{\theta}}.
\end{equation}
}
The estimate (\ref{eqno32}) in the case $\theta = \infty$    was proved in \cite{Temlyakov1980},  and in the case $1 \leq \theta < \infty$ in \cite{Romanyuk2011}.

\begin{remark}
  A  question concerning the orders of the quantity $E_{Q^{{\boldsymbol{\gamma}}}_n} ( B^{\boldsymbol{r}}_{\infty,\theta})_\infty$
%, $1 \leq \theta \leq \infty$,  
for $d > 2$  remains open. In particular, for the classes $H^{\boldsymbol{r}}_\infty$  see \cite[Open Problem 4.7]{Dung_Temlyakov_Ullrich2019}.  
\end{remark}

\begin{remark}
    Comparing Theorem \ref{Thm3} for  $p=1$ with Theorem D we conclude that for  $d\geq 2$, $1 \leq \theta  \leq \infty$  and $r_1 > 0$  it holds
\begin{equation*}
E_{Q^{{\boldsymbol{\gamma}}'}_n}(B^{\boldsymbol{r}}_{1,\theta})_{B_{1,1}} \asymp E_{Q^{{\boldsymbol{\gamma}}'}_n}(B^{\boldsymbol{r}}_{1,\theta})_1
\asymp 2^{-n r_1}n^{(\nu - 1)(1 - \frac{1}{\theta})}.
\end{equation*}

Similarly, comparing 
%(\ref{eqno27}) 
Theorem \ref{Thm3} for $p = \infty$ (see also (\ref{eqno27.gamma}))  with
%(\ref{eqno32}) 
Theorem E
we see that for   $d=2$, $1 \leq  \theta \leq \infty$  and $\boldsymbol{r}=(r_1, r_1)$, $r_1 > 0$, it holds
\begin{equation*}%\label{eqno33}
E_{Q^{{\boldsymbol{\gamma}}}_n}(B^{\boldsymbol{r}}_{\infty,\theta})_{B_{\infty,1}} \asymp E_{Q^{{\boldsymbol{\gamma}}}_n}(B^{\boldsymbol{r}}_{\infty,\theta})_\infty
\asymp 2^{-n r_1}n^{1 - \frac{1}{\theta}}.
\end{equation*}
\end{remark}

\begin{remark}
The orders of the quantities $\mathscr{E}_{Q^{{\boldsymbol{\gamma}}'}_n}(B^{\boldsymbol{r}}_{p,\theta})_{B_{p,1}}$   and   $\mathscr{E}_{Q^{{\boldsymbol{\gamma}}'}_n}(B^{\boldsymbol{r}}_{p,\theta})_{p}$, $p \in \{ 1, \infty \}$,  
$1\leq \theta \leq \infty$, $r_1>0$
in the case $d \geq 2$ 
%under the conditions of Theorem \ref{Thm3} on the parameters  $p, \theta, r_1$  
remain unknown.
%open. 
\end{remark}

In view of (\ref{eqno27}), (\ref{eqno27.gamma}) and (\ref{eqno32}),
we formulate the conjecture.

\textbf{Conjecture 1.} Let  $d>2$, $1 \leq   \theta \leq  \infty$.  Then for  $r_1 >0$ it holds
\begin{align*}
  E_{Q^{{\boldsymbol{\gamma}}'}_n}(B^{\boldsymbol{r}}_{\infty,\theta})_\infty & \asymp 2^{-n r_1}n^{(\nu - 1)(1 - \frac{1}{\theta})},
  \\
   E_{Q^{{\boldsymbol{\gamma}}}_n}(B^{\boldsymbol{r}}_{\infty,\theta})_\infty & \asymp 2^{-n r_1}n^{(d - 1)(1 - \frac{1}{\theta})}.
\end{align*}

To conclude this part of the paper, we consider one more relation between the parameters $p$   and  $q$.

\begin{theorem}\label{Thm4}
Let $d \geq 2$,  $1 \leq q < p \leq \infty$, $1 \leq  \theta \leq \infty$.  Then for $r_1 > 0$  it holds
\begin{equation}\label{eqno34}
\mathscr{E}_{Q^{{\boldsymbol{\gamma}}'}_n} ( B^{\boldsymbol{r}}_{p,\theta})_{B_{q,1}}    \asymp  E_{Q^{{\boldsymbol{\gamma}}'}_n} ( B^{\boldsymbol{r}}_{p,\theta})_{B_{q,1}}   \asymp  2^{ -n r_1} n^{(\nu - 1)(1 - \frac{1}{\theta})}.
\end{equation}
\end{theorem}
\begin{proof}
The upper estimates follow from Theorem \ref{Thm1}. Indeed, taking into account that
 $\|\cdot\|_{B_{1,1}} \ll  \|\cdot\|_{B_{q,1}}$, $1 \leq q < \infty$,
  it is sufficient to prove these estimates in the case $1 < q < p \leq \infty$.
 From the other side, due to the embeddings
  $B^{\boldsymbol{r}}_{\infty,\theta} \subset B^{\boldsymbol{r}}_{p,\theta} \subset B^{\boldsymbol{r}}_{q, \theta}$,  we can restrict ourself in considering the case
  $1 < p =q < \infty$, i.e., use the corresponding estimates obtained in Theorem \ref{Thm1}.

  As for the lower estimates in (\ref{eqno34}), we note that they follow from the
  corresponding 
  estimates for the Kolmogorov width
  %of the class    $B^{\boldsymbol{r}}_{p,\theta}$  in the space $B_{1,1}$ (see Remark \ref{Rem3} to Theorem \ref{Thm8}).
%The lower estimate of   $E_{Q^{{\boldsymbol{\gamma}}'}_n} ( B^{\boldsymbol{r}}_{\infty,\theta})_{B_{\infty,1}}$ follows from the corresponding estimate of the Kolmogorov widths
(see Remark \ref{Rem3} to Theorem~\ref{Thm8}), that will be obtained in 
%the next part of the paper 
Section~\ref{sec3:widths_en}.

Theorem \ref{Thm4} is proved.
\end{proof}

Let us recall the known statements for the corresponding approximation characteristics in the space $L_q$.

\textbf{Theorem F.}   {
 \it Let $d \geq 2$, $1  < q < p  \leq \infty$, $p^{\ast} = \min \{ 2,p \}$, $r_1 > 0$.   Then for $p^{\ast} < \theta \leq \infty$
 it holds
\begin{equation}\label{eqno35.1}
\mathscr{E}_{Q^{{\boldsymbol{\gamma}}'}_n} ( B^{\boldsymbol{r}}_{p,\theta})_q    \asymp E_{Q^{{\boldsymbol{\gamma}}'}_n} ( B^{\boldsymbol{r}}_{p,\theta})_q    \asymp  2^{ -n  r_1} n^{(\nu - 1)(\frac{1}{p^{\ast}} - \frac{1}{\theta})}.
\end{equation}
In the case $1 \leq  \theta \leq  p^{\ast}$, it holds
\begin{equation*}%\label{eqno35}
\mathscr{E}_{Q^{{\boldsymbol{\gamma}}}_n} ( B^{\boldsymbol{r}}_{p,\theta})_q    \asymp E_{Q^{{\boldsymbol{\gamma}}}_n} ( B^{\boldsymbol{r}}_{p,\theta})_q    \asymp  2^{ -n  r_1}.
\end{equation*}
}

For the classes $B^{\boldsymbol{r}}_{p,\infty}$  in  the  case  $2 \leq  q < p \leq \infty$  the  corresponding  upper  bounds  follow  from   the   upper   bounds  for the  case  $ 1 < q =p < \infty$ (see Theorem B)  and  the  lower  bounds  follow  from \cite{Temlyakov1980.1}.   In the  case  $1 < q < 2 \leq p < \infty$,  Theorem F  was  proved  in \cite{Dung1984,Galeev1984}.  In the case  $1 \leq  q < p \leq 2$,  the  proof  of the lower  bounds  required  a  new  technique, see \cite{Temlyakov1989}.  For $ 1 \leq \theta  < \infty$, the corresponding  statement was obtained in   \cite{Romanyuk1991}.

\textbf{Theorem G.}   {  \it
Let $d \geq 2$, 
$1  < p  \leq \infty$,  $1\leq \theta \leq \infty$.   Then for  $r_1 > 0$  it holds
\begin{equation}\label{eqno36}
\mathscr{E}_{Q^{{\boldsymbol{\gamma}}'}_n} ( B^{\boldsymbol{r}}_{p,\theta})_1    \asymp E_{Q^{{\boldsymbol{\gamma}}'}_n} ( B^{\boldsymbol{r}}_{p,\theta})_1    \asymp  2^{ -n  r_1} n^{(\nu - 1)(\frac{1}{p^*} - \frac{1}{\theta})_{+}}.
\end{equation}
}

In the case $1  < p  \leq 2$,   Theorem G   was proved in \cite{Temlyakov1989}  for $\theta = \infty$,  and in \cite{Galeev1990,Romanyuk2008}  for  $1 \leq \theta < \infty$.

For $2  < p  \leq \infty$ and $\theta = \infty$,  the  upper  bounds  in  Theorem G  follow from  the  upper  bounds  for $\mathscr{E}_{Q^{{\boldsymbol{\gamma}}'}_n} ( H^{\boldsymbol{r}}_2)_2$, see \cite{Bugrov1964}.   The   lower  bounds  are nontrivial. They   follow  from  the  corresponding  lower  bounds  for  the  Kolmogorov widths  $d_M(H^{\boldsymbol{r}}_{\infty}, L_1)$,  which,  as it  was  observed  in \cite{Belinskii1990},   are derived  from  the  lower  bounds  for the  entropy   numbers $\varepsilon_M(H^{\boldsymbol{r}}_\infty, L_1)$   from \cite{Temlyakov1990}.

If  $2  < p  \leq \infty$ and
$1 \leq \theta < \infty$, the upper estimate for the quantity $\mathscr{E}_{Q^{{\boldsymbol{\gamma}}'}_n} ( B^{\boldsymbol{r}}_{p,\theta})_1$
follows from Theorem B as $p=2$. It is due to the relation $\mathscr{E}_{Q^{{\boldsymbol{\gamma}}'}_n} ( B^{\boldsymbol{r}}_{p,\theta})_1  \ll \mathscr{E}_{Q^{{\boldsymbol{\gamma}}'}_n} ( B^{\boldsymbol{r}}_{2,\theta})_2$, $2 < p \leq \infty$.

The  lower estimates for the quantity $E_{Q^{{\boldsymbol{\gamma}'}}_n} ( B^{\boldsymbol{r}}_{p,\theta})_1 $,  $2  < p  \leq \infty$, in the case $1 \leq \theta \leq 2$  were obtained in \cite{Romanyuk2008}, and for  $2 < \theta < \infty$  they follow from the estimate for the Kolmogorov width    $d_M(B^{\boldsymbol{r}}_{\infty, \theta}, L_1)$, see \cite{Romanyuk2015}.

\begin{remark}
   Comparing the estimates (\ref{eqno34}) with (\ref{eqno35.1}) -- (\ref{eqno36}), we conclude that
%for the considered in Theorem \ref{Thm4} relations between the parameters $p$   and $q$ 
the corresponding approximation characteristics for the classes  $B^{\boldsymbol{r}}_{p,\theta}$,
$r_1>0$, $1\leq \theta\leq \infty$  
in the spaces  $B_{q,1}$ and  $L_q$ for $1\leq q< p\leq \infty$ and $d\geq 2$ coincide in order only either in the case $\theta = 1$  or  $\nu =1$. 
\end{remark}

\begin{remark}
It is worth noting, that in the univariate case these
%corresponding 
approximation characteristics in the spaces
$B_{q,1}$  and $L_q$ coincide in order for all values of the parameters $p$, $q$  and $\theta$, see \cite{Romanyuk_Romanyuk_Pozharska_Hembarska2023}.
\end{remark}

\begin{remark}\label{summung_up_best_approx}
    Summing up all of the results from Section \ref{sec2:App_shFs}, we get the following statement.

    Let $d\geq 2$, $1\leq \theta\leq \infty$. Then for $r>(1/p-1/q)_+$ it holds
$$
 E_{Q^{{\boldsymbol{\gamma}^*}}_n} ( B^{\boldsymbol{r}}_{p,\theta})_{B_{q,1}} 
  \asymp \mathscr{E}_{Q^{{\boldsymbol{\gamma}}^*}_n} ( B^{\boldsymbol{r}}_{p,\theta})_{B_{q,1}}
 \asymp  2^{ -n \left(  r_1-(\frac{1}{p}-\frac{1}{q})_+\right)} n^{(\mu^*-1)(1 - \frac{1}{\theta})},
$$
    where 
    \begin{itemize}
        \item in the case $1<p<q<\infty$ we put ${\boldsymbol{\gamma}}^* = \boldsymbol{\gamma}$, $\mu^* = \nu$;
       
        \item in the cases  $1<p=q<\infty$ and $1\leq q <p \leq \infty$ we put
        either ${\boldsymbol{\gamma}}^* = \boldsymbol{\gamma}'$, $\mu^* = \nu$  or ${\boldsymbol{\gamma}}^* = \boldsymbol{\gamma}$, $\mu^* = d$.
\end{itemize}
If $p=q\in \{1, \infty\}$, then 
    $$
     E_{Q^{{\boldsymbol{\gamma}}*}_n} ( B^{\boldsymbol{r}}_{p,\theta})_{B_{p,1}}   \asymp  2^{ -n  r_1} n^{(\mu^* - 1)(1 - \frac{1}{\theta})}
    $$ 
     with 
 either ${\boldsymbol{\gamma}}^* = \boldsymbol{\gamma}'$, $\mu^* = \nu$  or ${\boldsymbol{\gamma}}^* = \boldsymbol{\gamma}$, $\mu^* = d$.
    
%\noindent The limiting cases $q=\infty$, $1<p<\infty$ and $p=1$, $1<q\leq \infty$ remain open.
\end{remark}

\section{Widths and entropy numbers}\label{sec3:widths_en}

In this part of the paper, the main attention is paid to getting the exact order estimates for the Kolmogorov widths and entropy numbers for the classes  $B^{\boldsymbol{r}}_{p,\theta}$  in the space $B_{q,1}$. In addition, using the obtained estimates for the Kolmogorov widths in combination with the results from Section \ref{sec2:App_shFs}, we write down the exact order estimates for the linear widths of the investigated function classes in the same space.

As a result of research, we found out that the orders of the Kolmogorov and linear widths in the considered cases are realized by the subspace of trigonometric polynomials with spectrum in the step hyperbolic crosses.

Let us define the corresponding approximation  characteristics.

Let $\mathscr{Y}$  be  a  normed  space  with  the  norm $\|\cdot\|_{\mathscr{Y}}$,
$\mathscr{L}_M(\mathscr{Y})$  be  the  set  of all  subspaces  of  dimension  at   most    $M$ in  the  space  $\mathscr{Y}$ ,
 and  $W$   be  a centrally-symmetric set in $\mathscr{Y}$.

The quantity
$$
d_M(W, \mathscr{Y}) := \inf\limits_{L_M \in \mathscr{L}_M(\mathscr{Y})} \sup\limits_{w \in W} \inf\limits_{u \in L_M}\| w - u\|_{\mathscr{Y}}
$$
is called the  Kolmogorov $M$-width of the  set  $W$ in the space $\mathscr{Y}$.

The  width  $d_M(W, \mathscr{Y})$ was  introduced  in 1936  by   A.N. Kolmogorov \cite{Kolmogorov1936}.

Let  $ \mathscr{Y}$   and  $\mathscr{Z}$  be  normed  spaces  and  $L(\mathscr{Y},  \mathscr{Z})$  be  a  set  of  linear  continuous  mappings  of  $\mathscr{Y}$  into  $\mathscr{Z}$.

The quantity
$$
\lambda_M(W, \mathscr{Y}) :=  \mathop{\inf_{L_M \in \mathscr{L}_M(\mathscr{Y})}} \limits_{\Lambda \in L(\mathscr{Y}, L_M)} \sup\limits_{w \in W} \|w - \Lambda w\|_{\mathscr{Y}},
$$
where 
%the  lower  bound is  taken  over  the  all subsets  $L_M$  in  $\mathscr{L}_M(\mathscr{Y})$ 
%of  dimension at  most $M$ and all 
 $L(\mathscr{Y}, L_M)$ denotes linear  continuous  operators  that  map  from  $\mathscr{Y}$  into   $L_M$,  is  called  the  linear $M$-width. It was  introduced  in 1960   by  V.M. Ti\-kho\-mi\-rov \cite{Tikhomirov1960}.

From the  definitions  of  widths,  it follows that 
\begin{equation}\label{eqno38}
d_M(F, \mathscr{Y}) \leq \lambda_M(F, \mathscr{Y}).
\end{equation}

The  history of  investigation  of  widths  for  different  function  classes  of  periodic  multivariate functions  is given in the monographs  
\cite{Dung_Temlyakov_Ullrich2019,Romanyuk2012,Temlyakov1989_121p,Temlyakov1993,Temlyakov2018,Trigub_Belinsky2004}

%Let us define one more 

%approximation characteristics.

Let  $\mathscr{X}$  be    a Banach  space  and  let $B_{\mathscr{X}}(y, R)$  be  a ball  in $\mathscr{X}$   of  radius~$R$   centered  at  a point $y$, i.e.,
$$
B_{\mathscr{X}}(y, R) :=  \{x\in \mathscr{X}\colon \,\, \| {x}  -  y  \|_{\mathscr{X}} \leq R  \}.
$$
For a compact  set $A$  and $\varepsilon > 0$, we   introduce  the  number
\begin{equation}\label{eqno39}
N_{\varepsilon}(A, \mathscr{X}) : = \min \Big\{ n\colon \quad \exists y^1, \ldots, y^n \in \mathscr{X}\colon \quad
 A \subseteq \bigcup\limits^n_{j =1} B_{\mathscr{X}}(y^j, \varepsilon) \Big\}.
\end{equation}
Then the  quantity  (see. e.g., \cite{Kolmogorov_Tikhomirov1959})
$$
H_{\varepsilon}(A, \mathscr{X}) := \log N_{\varepsilon}(A,\mathscr{X})
$$
is  called  the   $\varepsilon$-entropy of the    set   $A$  with  respect to  the  Banach space $\mathscr{X}$   (here and   in what follows, $\log := \log_2$).

The  $\varepsilon$-entropy  of the  set $A$  is  closely  connected  to  
%the  notion  of  
its  entropy  numbers  $\varepsilon_k(A,\mathscr{X})$ \ (see. e.g., \cite{Hollig}):
$$
\varepsilon_k(A,\mathscr{X}) := \inf \Big\{\varepsilon\colon \quad \exists y^1, \ldots, y^{2^k}   \in \mathscr{X}\colon \quad
 A \subseteq \bigcup\limits^{2^k}_{j=1} B_{\mathscr{X}}(y^j, \varepsilon)\Big\}.
$$

Note that  directly  from  the  definitions  of the  quantities $H_{\varepsilon}(A, \mathscr{X})$  and $ \varepsilon_k (A, \mathscr{X})$,   we conclude that  if  $H_{\varepsilon}(A, \mathscr{X}) \leq  k$
  then   $\varepsilon_k(A, \mathscr{X}) \leq  \varepsilon$  and,  vice  versa,  the  estimate
  $\varepsilon_k(A, \mathscr{X}) \leq  \varepsilon$     yields  the  estimate   $H_{\varepsilon}(A, \mathscr{X}) \leq  k$.

In other  words,  if  $k  < H_{\varepsilon}(A, \mathscr{X})  \leq  k +1 $  then
$\varepsilon_{k+1}(A, \mathscr{X})  \leq \varepsilon \leq \varepsilon_k(A, \mathscr{X})$. These relations  enable  one  to  establish  estimates for  the  $\varepsilon$-entropy  $H_{\varepsilon}(A, \mathscr{X})$   from  the  estimates  for  the corresponding entropy  numbers
$\varepsilon_k (A, \mathscr{X})$.

The above defined characteristics on the classes of periodic multivariate functions $W^{\boldsymbol{r}}_{p,\alpha}$, $H^{\boldsymbol{r}}_p$  and $B^{\boldsymbol{r}}_{p,\theta}$
in the space $L_q$ were investigated in a number of papers. The detailed bibliography can be found in the monographs \cite{Dung_Temlyakov_Ullrich2019,Temlyakov2018,Trigub_Belinsky2004}.

In addition, we highlight the recent papers \cite{Mayer_Ullrich2020,Romanyuk2019,Temlyakov_Ullrich2020}, where the questions concerning estimates of the Kolmogorov width and entropy numbers of the mentioned
 classes of functions  in the space $L_q$ were investigated, and also \cite{Kashin_Temlyakov1999,Kashin_Temlyakov,Romanyuk_Yanchenko2022__UMZ_74,Romanyuk_Yanchenko2023}, where similar questions were considered on the classes of quasi-continuous functions  $Q C$. The norm in this space is close in some sense to the $L_{\infty}$-norm, and 
 it is connected in a certain way with the $B_{\infty,1}$-norm (see, e.g., \cite{Kashin_Temlyakov} for the details).

The  next  statement will play an important role in obtaining the estimates for the entropy numbers and Kolmogorov width of the classes  $B^{\boldsymbol{r}}_{p,\theta}$  in the space   $B_{q,1}$. It is a  corollary  from the  Carl's  inequality (see, e.g., \cite{Carl1981}).

\textbf{Lemma B.} \cite{Temlyakov1996_East_J}. {\it
 Let  $A$   be  a  compact  set  in a  separable  Banach  space $\mathscr{X}$. Assume  that  for  a  pair  of  numbers  $(a,b)$,  where either
 $a > 0$, $b \in \mathbb{R}$  or $a = 0$, $b < 0$,  the  relations
$$
d_m(A, \mathscr{X})  \ll m^{-a}(\log m)^b, \qquad \varepsilon_m (A, \mathscr{X})  \gg m^{-a}(\log m)^b
$$
are  true.  Then, it holds
$$
\varepsilon_m(A, \mathscr{X}) \asymp d_m(A, \mathscr{X}) \asymp     m^{-a}(\log m)^b.
$$
}

Before we formulate and prove the results, let us introduce some further notation, taken from \cite{Temlyakov1990}.

For all ${\boldsymbol{s}}
%=(s_1,\dots,s_d)
\in \mathbb{N}^d$, let us define
$$
\overline{\rho}({\boldsymbol{s}})  := \{ \boldsymbol{k}
%= (k_1, \ldots, k_d)
\in \mathbb{N}^d \colon \  2^{s_j -1} \leq  k_j < 2^{s_j}, \ j = 1, \dots, d \},
$$
and  consider the set of 
%real 
trigonometric polynomials
$$
T(\overline{\rho}({\boldsymbol{s}})) := \Big\{ t\colon \ t(\boldsymbol{x}) = \sum\limits_{\boldsymbol{k} \in \overline{\rho}({\boldsymbol{s}})} \widehat{t}(\boldsymbol{k}) e^{i(\boldsymbol{k},\boldsymbol{x})}, \ \boldsymbol{x} \in \mathbb{R}^d  
\Big\},
$$
where $\widehat{t}(\boldsymbol{k})$ are the Fourier coefficients of $t$.
%Note,  that 

Every 
%polynomial 
$t \in T(\overline{\rho}({\boldsymbol{s}}))$, $s_j \geq 2$, $j = 1, \dots, d$,   can  be  represented  in  the  form
$$
t(\boldsymbol{x}) = e^{i(\boldsymbol{k}^{\boldsymbol{s}}, \boldsymbol{x})} t^1(\boldsymbol{x}),
$$
where    $\boldsymbol{k}^{\boldsymbol{s}} = (k^{s_1}_1, \ldots, k^{s_d}_d)$, $k^{s_j}_j = 2^{s_j -1} + 2^{s_j - 2}$, $j = 1, \dots, d$,   and  $t^1(\boldsymbol{x})$    is    a   polynomial   of   degree $2^{s_j -2}$
   with  respect  to  the  variable $x_j$, $j = 1, \dots, d$.

For every ${\boldsymbol{m}}
%= (m_1, \ldots, m_d)$
\in \mathbb{Z}_{+}^d$,
%where $\mathbb{Z}_{+}^d
i.e., with nonnegative integer coordinates,
%$, $m_j \in \mathbb{Z}_{+}$
we 
put
%denote  by   $R T(m)$  the  set  of  real  trigonometric  polynomials  $t$  of the  form
$$
R T({\boldsymbol{m}}) := \Bigg\{ t\colon \quad t(\boldsymbol{x}) = \mathop{ \sum_{|k_j| \leq m_j}}\limits_{j = 1, \dots, d} \widehat{t}(\boldsymbol{k}) e^{i(\boldsymbol{k},\boldsymbol{x}) } , \ \boldsymbol{x} \in \mathbb{R}^d \Bigg\}.
$$
Let $\boldsymbol{2}=(2,\dots, 2)\in \mathbb{N}^d$ and
% $T'(\overline{\rho}(s))$ be  the  set  of  trigonometric  polynomials  of  the    form
$$
T'(\overline{\rho}({\boldsymbol{s}})) := \Big\{ t\colon \ t(\boldsymbol{x}) =  e^{i(\boldsymbol{k}^{\boldsymbol{s}},\boldsymbol{x})} t^1(\boldsymbol{x}), \,\, t^1 \in R T(2^{{\boldsymbol{s}} - \boldsymbol{2}})\Big\}.
$$

For  even  $n\in \mathbb{N}$,
% $n\geq 2d$,  
we    define  the  sets
\begin{gather*}
\Omega_n := \{  {\boldsymbol{s}}\in \mathbb{N}^d \colon \, 
%\|{\boldsymbol{s}}\|_1
(\boldsymbol{s},\boldsymbol{1})
= n, \, s_j  \mbox{ are even  numbers}, \, j = 1, \dots, d\}, \
 Q'_n := \bigcup\limits_{{\boldsymbol{s}} \in \Omega_n} \overline{\rho}({\boldsymbol{s}}),
\\
T'(Q'_n) := \bigcup\limits_{{\boldsymbol{s}} \in \Omega_n} T'(\overline{\rho}({\boldsymbol{s}}))  = \{ t\colon \quad t(\boldsymbol{x}) = \sum\limits_{{\boldsymbol{s}} \in \Omega_n} e^{i(\boldsymbol{k}^{\boldsymbol{s}}, \boldsymbol{x})}t^1_{\boldsymbol{s}}(\boldsymbol{x}), \quad t^1_{\boldsymbol{s}} \in R T(2^{{\boldsymbol{s}} - \boldsymbol{2}})\},
\\
T'(Q'_n)_{\infty} := \{ t \in T'(Q'_n)\colon \quad   \|t^1_{\boldsymbol{s}}\|_{\infty} \leq 1\}.
\end{gather*}

In what follows, we 
%formulate a statement concerning 
get the lower estimate for the entropy numbers of the classes  $B^{\boldsymbol{r}}_{\infty, \theta}$,
$1\leq \theta\leq \infty$,  in the space $B_{1,1}$. It is an expansion
%to these classes 
of the
 corresponding
 %statement
result by V.N. Temlyakov \cite{Temlyakov1990}
%which was obtained by 
for the classes  $H^{\boldsymbol{r}}_{\infty}$.

\begin{theorem}\label{Thm5}
 Let  $ r_1 > 0$   and  $1 \leq \theta \leq \infty$.   Then it holds
\begin{equation*}%\label{eqno40}
\varepsilon_M ( B^{\boldsymbol{r}}_{\infty,\theta}, B_{1,1}) \gg   M^{ - r_1} (\log^{\nu - 1} M)^{ r_1 + 1 - \frac{1}{\theta}}.
\end{equation*}
\end{theorem}
\begin{proof}
In the case $\theta = \infty$, as it was noted above, the estimate was obtained in  \cite{Temlyakov1990}. Therefore,
 we assume that 
 $1\leq \theta <\infty$. Besides, it  is  sufficient  to  consider the  case $\nu = d$.

We will use some relations from  the proof of Theorem 2.2 in \cite{Temlyakov1990}, in particular, the estimate
\begin{equation}\label{eqno41}
\varepsilon_M ( T'(Q'_n)_\infty \, 2^{-n r_1}, L_2) \gg   M^{ - r_1} (\log^{d - 1} M)^{ r_1 +  \frac{1}{2}},
\end{equation}
where $M = |Q^{'}_n| \asymp 2^n n^{d - 1}$.

Since the  embedding
$$
T'(Q'_n)_{\infty}\, 2^{-n r_1} \subset C_7 H^{\boldsymbol{r}}_{\infty}, \quad C_7 > 0,
$$
is true,  for   $f \in T'(Q'_n)_{\infty} \, 2^{-n r_1}$
it holds $\|A_{\boldsymbol{s}}(f)\|_{\infty} \ll 2^{-({\boldsymbol{s}},r)}$, $s \in  \Omega_n$
(see (\ref{norm_via_A_s_infty})).

Therefore,  for  $f \in  T'(Q'_n)_{\infty} \, 2^{-n r_1}$ from (\ref{norm_via_A_s_notinfty}) and (\ref{norm_As_1}) we  get
\begin{align*}
\|f\|_{B^{\boldsymbol{r}}_{\infty, \theta}} & \asymp   \Bigg(  \sum\limits_{{\boldsymbol{s}} \in \Omega_n}  2^{({\boldsymbol{s}},{\boldsymbol{r}})\theta} \|A_{\boldsymbol{s}}(f)\| ^{\theta}_{\infty} \Bigg)^{\frac{1}{\theta}}
\ll \Bigg(  \sum\limits_{{\boldsymbol{s}} \in \Omega_n}  2^{({\boldsymbol{s}},{\boldsymbol{r}})\theta} \|A_{\boldsymbol{s}}\|^{\theta}_1\,\,  \|f\|^{\theta}_{\infty} \Bigg)^{\frac{1}{\theta}}
\\
&  \ll \Bigg(  \sum\limits_{{\boldsymbol{s}} \in \Omega_n}  2^{({\boldsymbol{s}},{\boldsymbol{r}})\theta}\,\, \|f\|^{\theta}_{\infty} \Bigg)^{\frac{1}{\theta}} \ll   \Bigg(  \sum\limits_{{\boldsymbol{s}} \in \Omega_n}  1 \Bigg)^{\frac{1}{\theta}} \asymp n^{\frac{d -1}{\theta}},
\end{align*}
and thus
\begin{equation}\label{eqno42}
T'(Q'_n)_{\infty} \, 2^{-n r_1} n^{ - \frac{d -1}{\theta}} \subset  C_8 B^{\boldsymbol{r}}_{\infty, \theta}, \quad 1 \leq \theta < \infty, \quad  C_8 > 0. \end{equation}
Hence, according  to (\ref{eqno41}), in view of $M \asymp 2^n n^{d - 1}$, i.e., $n \asymp \log M$,  we  can  write
\begin{align}
\varepsilon_M ( T'(Q'_n)_{\infty}\,\,  2^{-n r_1}\,\,  n^{ - \frac{d -1}{\theta}}, L_2) &
%\gg  M^{ - r_1} (\log^{d - 1} M)^{ r_1 +  \frac{1}{2}}\,\, n^{ - \frac{d -1}{\theta}}
%\nonumber
%\\&
\gg M^{ - r_1} (\log^{d - 1} M)^{ r_1 +  \frac{1}{2} - \frac{1}{\theta}}
\nonumber
\\
& 
\asymp 2^{-n r_1} n^{(d - 1)(\frac{1}{2} - \frac{1}{\theta})}.   \label{eqno43}
\end{align}
%where $M \asymp 2^n n^{d-1}$.

Let $A$   be  a  compact  set  and  $M_{\varepsilon}(A, \mathscr{X})$    be   the  maximum  number  of  points $x^j \in A$  such  that  $\|x^i - x^j\|_{\mathscr{X}} > \varepsilon$, $i \neq j$; $N_{\varepsilon}(A, \mathscr{X})$  be  a  number  defined  by  the formula  (\ref{eqno39}).

It is well known that 
\begin{equation}\label{eqno44}
N_\varepsilon (A, \mathscr{X}) \leq M_\varepsilon (A, \mathscr{X}) \leq N_{\varepsilon/2} (A, \mathscr{X}).
\end{equation}
The relations (\ref{eqno43}) and (\ref{eqno44}) yield that in   $T'(Q'_n)_{\infty}\,  2^{-n r_1} n^{ - \frac{d -1}{\theta}}$  there exist $2^M$  functions  $\{ f_j \}^{2^M}_{j=1}$  such that for $i \neq j$  it holds
\begin{equation}\label{eqno45}
\|f_i - f_j\|_2 \gg 2^{-n r_1} n^{(d - 1)(\frac{1}{2} - \frac{1}{\theta})}.
\end{equation}
Further
%in view of (\ref{eqno45}), 
we  
%can easily get the estimate
show that
\begin{equation}\label{eqno46}
\|f_i - f_j\|_{B_{1,1}} \gg 2^{-n r_1} n^{(d - 1)(1 - \frac{1}{\theta})}.
\end{equation}

Indeed, taking into account (\ref{eqno45}) and the definition (\ref{norm_Bq1_via_As}) of the norm in the space $B_{1,1}$, for $f \in T'(Q'_n)_{\infty}$ we obtain
\begin{align}
n^{d -1} \ll \|f\|^2_2 & = \sum\limits_{{\boldsymbol{s}} \in \Omega_n} \, \| \sum\limits_{\|{\boldsymbol{s}} - {\boldsymbol{s}}'\|_{\infty} \leq 1} A_{{\boldsymbol{s}}'}(f) \|^2_2 \leq
  \sum\limits_{{\boldsymbol{s}} \in \Omega_n}  \sum\limits_{\|{\boldsymbol{s}} - {\boldsymbol{s}}'\|_{\infty} \leq 1} \|A_{{\boldsymbol{s}}'}(f) \|^2_2
\nonumber\\
& \leq  \sum\limits_{{\boldsymbol{s}} \in \Omega_n}  \sum\limits_{\|{\boldsymbol{s}} - {\boldsymbol{s}}'\|_{\infty} \leq 1} \|A_{{\boldsymbol{s}}'}(f) \|_1 \, \|A_{{\boldsymbol{s}}'}(f)\|_\infty
\nonumber\\
& \leq  \mathop{\max_{{\boldsymbol{s}} \in \Omega_n}} \limits_{\|{\boldsymbol{s}} - {\boldsymbol{s}}'\|_{\infty} \leq 1} \|A_{{\boldsymbol{s}}'}(f) \|_\infty  \sum\limits_{{\boldsymbol{s}}\in \mathbb{N}^d} \|A_{{\boldsymbol{s}}}(f) \|_1 \ll \|f\|_{B_{1,1}}.
\label{eqno47}
\end{align}
Combining (\ref{eqno45}) with (\ref{eqno47}), we get (\ref{eqno46}).

Hence, from  (\ref{eqno42}) and (\ref{eqno46}), in view of $M \asymp 2^n n^{d - 1}$, we obtain the estimate
\begin{align*}
\varepsilon_M ( B^{\boldsymbol{r}}_{\infty,\theta}, B_{1,1})& \gg
\varepsilon_M (T'(Q'_n)_{\infty} \, 2^{-n r_1} n^{ - \frac{d -1}{\theta}}, B_{1,1})
%\\
%&    =  2^{-n r_1} n^{ - \frac{d -1}{\theta}}  \varepsilon_M (T'(Q'_n)_{\infty}, B_{1,1})
\gg   
%2^{-n r_1} \, n^{ - \frac{d -1}{\theta}} n^{d - 1}
2^{-n r_1} n^{(d - 1)(1 - \frac{1}{\theta})}
\\
&    \asymp     M^{ - r_1} (\log^{d - 1} M)^{ r_1 + 1 - \frac{1}{\theta}}.
\end{align*}

Theorem \ref{Thm5} is proved.
\end{proof}

In the next part, we prove the exact-order estimates for the Kolmogorov, linear widths and entropy numbers of the classes $B^{\boldsymbol{r}}_{p,\theta}$  in the space  $B_{q,1}$.  

\begin{theorem}\label{Thm8}
Let  $d \geq 2$, $r_1 > 0$, $1 \leq q \leq  p \leq \infty$, $1 \leq \theta \leq \infty$.  Then  it holds
\begin{equation}\label{eqno57}
\varepsilon_M ( B^{\boldsymbol{r}}_{p,\theta},  B_{q,1}) \asymp  d_M ( B^{\boldsymbol{r}}_{p, \theta}, B_{q,1})   \asymp  \lambda_M ( B^{\boldsymbol{r}}_{p, \theta}, B_{q,1})  \asymp
  M^{ - r_1} (\log^{\nu - 1} M)^{ r_1 + 1 - \frac{1}{\theta}}.
\end{equation}
\end{theorem}
\begin{proof} First we get the right order for  $d_M ( B^{\boldsymbol{r}}_{p, \theta}, B_{q,1})$ and $\varepsilon_M ( B^{\boldsymbol{r}}_{p,\theta},  B_{q,1})$ by using 
 Lemma B.

Indeed, on the one hand, the upper estimate for the Kolmogorov width    follows from Theorems~\ref{Thm1}, \ref{Thm3}  and \ref{Thm4} under the condition  $M \asymp 2^n n^{\nu - 1}$:
\begin{align*}
d_M ( B^{\boldsymbol{r}}_{p, \theta}, B_{q,1}) & \ll  E_{Q^{{\boldsymbol{\gamma}} '}_n}(B^{\boldsymbol{r}}_{p, \theta})_{B_{q,1}} \asymp 2^{- n r_1} n^{(\nu - 1)(1- \frac{1}{\theta})} \nonumber\\
&
\asymp M^{ - r_1} (\log^{\nu - 1} M)^{ r_1 + 1 - \frac{1}{\theta}}.
%\label{eqno58}
\end{align*}

On the other hand, 
the lower estimate for the entropy numbers  is a corollary from  Theorem \ref{Thm5}:
\begin{align*}
\varepsilon_M ( B^{\boldsymbol{r}}_{p,\theta}, B_{q,1})   & \geq \varepsilon_M ( B^{\boldsymbol{r}}_{p,\theta}, B_{1,1})   \geq \varepsilon_M ( B^{\boldsymbol{r}}_{\infty,\theta}, B_{1,1})
\nonumber\\
& \gg M^{ - r_1} (\log^{\nu - 1} M)^{ r_1 + 1 - \frac{1}{\theta}}.
%\label{eqno50}
\end{align*}

%Therefore, taking Lemma B into account with respect to the relations (\ref{eqno58})  and  (\ref{eqno50}), we obtain  %(\ref{eqno57}).

As for $\lambda_M ( B^{\boldsymbol{r}}_{p, \theta}, B_{q,1})$, here we note, that for 
$(p, q)\in \{(1, 1), (\infty, \infty)\}$
%$p=q\in \{1, \infty\}$
the upper estimate for the quantity  of the best approximation
$E_{Q^{{\boldsymbol{\gamma}}'}_n} ( B^{\boldsymbol{r}}_{p,\theta})_{B_{p,1}}$ from Theorem \ref{Thm3}
 is realized by a linear method, based on the polynomials (\ref{eqno27.1}).
    The corresponding lower estimate is due to the relation (\ref{eqno38}).

Theorem \ref{Thm8} is proved.
\end{proof}

\begin{remark}\label{Rem3}
The estimate (\ref{eqno57}) for the Kolmogorov width yields the lower bounds for $E_{Q^{{\boldsymbol{\gamma}} '}_n}(B^{\boldsymbol{r}}_{\infty, \theta})_{B_{\infty,1}}$ in Theorem \ref{Thm3} and for $
%\mathscr{E}_{Q^{{\boldsymbol{\gamma}} '}_n}(B^{\boldsymbol{r}}_{p, \theta})_{B_{q,1}} \geq  
E_{Q^{{\boldsymbol{\gamma}} '}_n}(B^{\boldsymbol{r}}_{p, \theta})_{B_{q,1}}$, $1 \leq q < p \leq \infty$, in Theorem \ref{Thm4}. 
Indeed, for $M \asymp 2^n n^{\nu - 1}$ it holds
$$
%\mathscr{E}_{Q^{{\boldsymbol{\gamma}} '}_n}(B^{\boldsymbol{r}}_{p, \theta})_{B_{q,1}} \geq    
E_{Q^{{\boldsymbol{\gamma}} '}_n}(B^{\boldsymbol{r}}_{p, \theta})_{B_{q,1}}  \gg
 d_M ( B^{\boldsymbol{r}}_{p, \theta}, B_{q,1})  \asymp 2^{- n r_1} n^{(\nu - 1)(1- \frac{1}{\theta})}.
 $$
 \end{remark}

Below we 
compare 
%the obtained
the result of Theorem  \ref{Thm8}
 with the estimates of the corresponding quantities in the space $L_q$.
We will consider several cases.

First, let $p=q$.

\begin{remark}
    In the case $p=q=1$, the question about the orders of the quantities $d_M(B^{\boldsymbol{r}}_{1,\theta}, L_1)$,  $\lambda_M(B^{\boldsymbol{r}}_{1,\theta}, L_1)$ and $\varepsilon_M ( B^{\boldsymbol{r}}_{1,\theta}, L_1)$  for  $d \geq 2$, $1\leq \theta\leq \infty$  remains open.  In particular, for the 
    %width  $d_M(H^{\boldsymbol{r}}_1, L_1)$ 
    classes $H^{\boldsymbol{r}}_1$ (see, e.g., \cite[Figures~4.3 and 4.5]{Dung_Temlyakov_Ullrich2019}).
\end{remark}

 Case $p=q= \infty$.  Here the following statements are known for the case $d=2$.

 \textbf{Theorem H}. \cite{Romanyuk2019}   {  \it
 Let  $ d = 2$, ${\boldsymbol{r}} = (r_1, r_1)$, $r_1 > 0$, $1 \leq \theta <  \infty$.  Then it holds
\begin{equation}\label{eqno51}
\varepsilon_M ( B^{\boldsymbol{r}}_{\infty,\theta}, L_{\infty}) \asymp  d_M ( B^{\boldsymbol{r}}_{\infty,\theta}, L_{\infty}) 
%\asymp  \lambda_M ( B^{\boldsymbol{r}}_{\infty,\theta}, L_{\infty}) 
\asymp   M^{ - r_1} (\log  M)^{ r_1 + 1 - \frac{1}{\theta}}.
\end{equation}
}

%The right order for the entropy numbers and Kolmogorov widths for the classes $B^{\boldsymbol{r}}_{\infty,\theta}$, %$1 \leq \theta <  \infty$, was obtained in \cite{Romanyuk2019} . 
%As for the classes $H^{\boldsymbol{r}}_{\infty}$, we note the

 \textbf{Theorem I}.   {  \it
 Let  $d = 2$,  ${\boldsymbol{r}} = (r_1, r_1)$, $r_1 > 0$.  Then it holds
\begin{equation}\label{eqno52}
\varepsilon_M ( H^{\boldsymbol{r}}_\infty,  L_{\infty}) \asymp  d_M ( H^{\boldsymbol{r}}_{\infty}, L_{\infty})  \asymp   M^{ - r_1} (\log  M)^{ r_1 + 1}.
\end{equation}
}

 The  upper  estimate  for  $d_M ( H^{\boldsymbol{r}}_{\infty}, L_{\infty})$  is  realized by approximating  functions  from  the  classes  $H^{\boldsymbol{r}}_\infty$  by  trigonometric  polynomials  with  harmonics  from  the hyperbolic  cross \cite[p. 55]{Temlyakov1989_121p}.  For the  corresponding  lower  estimate  for  the entropy  numbers  $\varepsilon_M ( H^{\boldsymbol{r}}_\infty,  L_{\infty})$,  see \cite{Temlyakov1995_Complexity}.  Applying  Lemma B,  we get (\ref{eqno52}).

\begin{corollary}\label{cor4.5}
 Let    $d = 2$, ${\boldsymbol{r}} = (r_1, r_1)$, $r_1 > 0$, $1 \leq  \theta \leq \infty$. Then    it holds
\begin{equation}\label{eqno63}
\lambda_M ( B^{\boldsymbol{r}}_{\infty,\theta},  L_{\infty}) \asymp   M^{ - r_1} (\log M)^{ r_1 + 1 - \frac{1}{\theta}}.
\end{equation}
\end{corollary}
\begin{proof}
As in Theorem I, the  upper  estimate in Corollary~\ref{cor4.5}
%for $\lambda_M ( B^{\boldsymbol{r}}_{\infty,\theta},  L_{\infty}) $  
is  realized  by  approximating functions  from  the  respective classes 
%$B^{\boldsymbol{r}}_{\infty, \theta}$ 
by  trigonometric  polynomials  with  harmonics  from  the hyperbolic  crosses (see
\cite[p. 55]{Temlyakov1989_121p}
for the case $\theta = \infty$  and  \cite{Romanyuk2011}  for  $1 \leq \theta < \infty$).

The lower estimate in (\ref{eqno63}) follows from (\ref{eqno51}) and  (\ref{eqno52}).
\end{proof}

To complement the already mentioned results, we formulate the known statement for the quantity $d_M(B^{\boldsymbol{r}}_{\infty,1}, L_\infty)$ in the case $d \geq 2$.

  \textbf{Theorem J}.  {  \it
  Let  $d \geq 2$, $r_1 > 0$.  Then it holds
\begin{equation}\label{eqno53}
 d_M ( B^{\boldsymbol{r}}_{\infty, 1}, L_{\infty})  \asymp  \lambda_M ( B^{\boldsymbol{r}}_{\infty, 1}, L_{\infty})  \asymp M^{ - r_1} (\log^{\nu - 1} M)^{ r_1}.
 \end{equation}
 }
The exact order of the Kolmogorov widths in (\ref{eqno53}) was obtained in \cite{Romanyuk2005_UMB}. 
The corresponding estimate for $\lambda_M ( B^{\boldsymbol{r}}_{\infty, 1}, L_{\infty})$ is due to 
the fact that the upper estimate for $d_M ( B^{\boldsymbol{r}}_{\infty, 1}, L_{\infty})$ is realized by a linear method.

\begin{remark}
    Comparing the estimates (\ref{eqno51}), (\ref{eqno52}), (\ref{eqno63}) and  (\ref{eqno53}) with the result of Theorem \ref{Thm8} for  $p = \infty$,
we see that the orders of the respective characteristics of the classes $B^{\boldsymbol{r}}_{\infty, \theta}$  in the spaces $B_{\infty, 1}$  and $L_\infty$ coincide. 

The question about the orders of the quantities $d_M(B^{\boldsymbol{r}}_{\infty,\theta}, L_\infty)$, \newline $\lambda_M ( B^{\boldsymbol{r}}_{\infty,\theta},  L_{\infty})$  and
$\varepsilon_M ( B^{\boldsymbol{r}}_{\infty,\theta}, L_\infty)$  in the case  $d \geq 3$   and    $1 < \theta \leq \infty$  still remains open.
\end{remark}

In  \cite[Theorem 4.4]{Cobos_Kuhn_Sickel_2019} it was shown that the approximation numbers for the embeddings of the periodic Sobolev spaces of functions with dominating mixed smoothness into the spaces 
$L_\infty$ and $B_{\infty,1}$ coincide. Based on the above results, we formulate a conjecture.

\textbf{Conjecture 2.} For $d\geq 2$, $r_1>0$,  $1 \leq \theta \leq \infty$ it holds
$$
\lambda_M ( B^{\boldsymbol{r}}_{\infty,\theta},  L_{\infty}) \asymp \lambda_M ( B^{\boldsymbol{r}}_{\infty,\theta},  B_{\infty,1}) \asymp   M^{ - r_1} (\log^{\nu - 1} M)^{ r_1 + 1 - \frac{1}{\theta}}.
$$

Case $2 \leq p=q < \infty$. 
%We prove the statement which is an analog to Theorem \ref{Thm8} as $p=q$ for the space $L_p$ under certain restrictions on the parameters  $p$  and   $\theta$.

\begin{theorem}\label{Thm7}
Let  $d \geq 2$, $r_1 > 0$, $2 \leq p < \infty$, $2 \leq \theta \leq \infty$.  Then     it holds
\begin{equation}\label{eqno54}
\varepsilon_M ( B^{\boldsymbol{r}}_{p,\theta},  L_p) \asymp  d_M ( B^{\boldsymbol{r}}_{p, \theta}, L_p) 
\asymp  \lambda_M ( B^{\boldsymbol{r}}_{p, \theta}, L_p)
\asymp 
M^{ - r_1} (\log^{\nu - 1} M)^{ r_1 +  \frac{1}{2} - \frac{1}{\theta}}.
\end{equation}
\end{theorem}

\begin{proof} 
Let first $2\leq \theta < \infty$. Then, on the one hand, Theorem B as $M \asymp 2^n n^{\nu - 1}$ yields
\begin{align}
d_M ( B^{\boldsymbol{r}}_{p, \theta}, L_p)  &\ll  \mathscr{E}_{Q^{{\boldsymbol{\gamma}} '}_n}(B^{\boldsymbol{r}}_{p, \theta})_p \asymp 2^{- n r_1} n^{(\nu - 1)(\frac{1}{2} - \frac{1}{\theta})} \nonumber \\
&\asymp M^{ - r_1} (\log^{\nu - 1} M)^{ r_1 +  \frac{1}{2} - \frac{1}{\theta}}.
\label{eqno55}
\end{align}

On the other hand, according to  \cite[Thm. 3]{Romanyuk2015}, it holds
\begin{equation}\label{eqno56}
\varepsilon_M ( B^{\boldsymbol{r}}_{p,\theta},  L_p)  \geq \varepsilon_M ( B^{\boldsymbol{r}}_{\infty,\theta},  L_1) \gg
 M^{ - r_1} (\log^{\nu - 1} M)^{ r_1 +  \frac{1}{2} - \frac{1}{\theta}}.
 \end{equation}

Using Lemma B with respect to the estimates (\ref{eqno55})  and (\ref{eqno56}), we get
$$
\varepsilon_M ( B^{\boldsymbol{r}}_{p,\theta},  L_p)   \asymp  d_M ( B^{\boldsymbol{r}}_{p, \theta}, L_p)  \asymp
 M^{ - r_1} (\log^{\nu - 1} M)^{ r_1 +  \frac{1}{2} - \frac{1}{\theta}}, \quad 2\leq \theta < \infty.
$$

Let $\theta = \infty$. Then we 
%obtain (\ref{eqno54}) by 
use Lemma B  with respect to the estimates
$$
\varepsilon_M ( H^{\boldsymbol{r}}_{\infty},  L_1) \gg M^{ - r_1} (\log^{\nu - 1} M)^{ r_1 +  \frac{1}{2}}
$$
and
$$
d_M ( H^{\boldsymbol{r}}_p,  L_p)   \ll    M^{ - r_1} (\log^{\nu - 1} M)^{ r_1 +  \frac{1}{2}}
$$
(see, e.g., \cite{Temlyakov1989}  and \cite[Ch. 3, $\S$\,2]{Temlyakov1989_121p}, respectively).
%, we obtain
% (\ref{eqno54})  for   $\theta  = \infty$.

As for the linear widths, the respective estimate is due to the relation
$$
d_M ( B^{\boldsymbol{r}}_{p, \theta}, L_p)  \leq \lambda_M ( B^{\boldsymbol{r}}_{p, \theta}, L_p)
\ll  \mathscr{E}_{Q^{{\boldsymbol{\gamma}} '}_n}(B^{\boldsymbol{r}}_{p, \theta})_p
$$
under the condition $M \asymp 2^n n^{\nu - 1}$.

Theorem \ref{Thm7} is proved.
\end{proof}

Let further $p\neq q$. The following statement is known.

\textbf{Theorem K} \cite{Romanyuk2015}.   {  \it
Let   $d \geq 2$, $r_1 > 0$, $2 \leq  p \leq \infty$, $1 \leq  q < p$, $2 \leq \theta \leq \infty$. Then   it holds
\begin{equation}\label{eqno60}
\varepsilon_M ( B^{\boldsymbol{r}}_{p,\theta},  L_q) \asymp  d_M ( B^{\boldsymbol{r}}_{p, \theta}, L_q)  \asymp
 M^{ - r_1} (\log^{\nu - 1} M)^{ r_1 + \frac{1}{2} - \frac{1}{\theta}}.
 \end{equation}
}

 In the case $\theta = \infty$, the estimates for the quantity $d_M ( H^{\boldsymbol{r}}_p, L_q)$  with respective comments can be found in
 \cite[Thms. 4.3.10 -- 4.3.13]{Dung_Temlyakov_Ullrich2019}. The lower estimate for the  entropy numbers $\varepsilon_M ( H^{\boldsymbol{r}}_{\infty},  L_1)$ was obtained in~\cite{Temlyakov1990}.
 
From Theorems F and K, we obtain the following.
\begin{corollary}\label{cor6}
Let $d \geq 2$, $r_1 > 0$, $2 \leq  p \leq  \infty$, $1 \leq q < p$, $2 \leq  \theta \leq \infty$.  Then  it holds
\begin{equation}\label{eqno66}
\lambda_M ( B^{\boldsymbol{r}}_{p,\theta},  L_q) \asymp   M^{ - r_1} (\log^{\nu - 1} M)^{ r_1 + \frac{1}{2} - \frac{1}{\theta}}.
\end{equation}
\end{corollary}

\begin{remark}
   Comparing   (\ref{eqno57})  with   (\ref{eqno54}), (\ref{eqno60}) and (\ref{eqno66}) , we conclude that 
 under the conditions of Theorems~\ref{Thm7}, K and Corollary \ref{cor6}
on the parameters $p$, $q$  and $\theta$, the corresponding approximation characteristics of the classes  $B^{\boldsymbol{r}}_{p,\theta}$   in the spaces $L_q$  and $B_{q,1}$ differ in order.
 
 For the remaining values of the parameters $p$, $q$ and $\theta$,
that fulfil the conditions of Theorem \ref{Thm8}, the orders of the quantities $d_M ( B^{\boldsymbol{r}}_{p,\theta},  L_q)$, $\lambda_M ( B^{\boldsymbol{r}}_{p,\theta},  L_q)$ and
$\varepsilon_M ( B^{\boldsymbol{r}}_{p, \theta},  L_q)$, to the best of our knowledge, remain open.
\end{remark}

\vskip4mm

\noindent\textbf{Acknowledgement.} KP would like to acknowledge support by the Philipp Schwartz Fellowship of the Alexander von Humboldt Foundation.




\begin{thebibliography}{00}


\bibitem{Amanov1965} 
% 15
  T.I. Amanov,
  Representation and  imbedding  theorems for  function  spaces $S^{(r)}_{p, \theta}B(\mathbb{R}_n)$
  and  $S^{(r)}_{p, \theta}\ast B$, $(0 \leq x_j \leq 2\pi$; $j = 1, \ldots, n)$,
   Tr. Mat. Inst. Akad. Nauk SSSR  77 (1965) 5--34 (in Russian).

 \bibitem{Belinskii1990}
 %33
   E.S. Belinskii,
   Asymptotic characteristics of  classes  of  functions with   conditions on the mixed  derivative (mixed difference).
   In: Studies in the theory of functions of several real variables, Yaroslav State Univ. (Y.~Brudnyi, Ed.), Yaroslavl (1990) 22--37 (in Russian).


\bibitem{Belinsky1998} 
%3
 E.S. Belinsky,
  Estimates of  entropy numbers  and Gaussian measures  for  classes of   functions  with   bounded  mixed  derivative,
  J. Approx. Theory 93 (1998) 114--127.


\bibitem{Besov1961} 
% 16
 O.V. Besov,
 Investigation of   one  family  of  functional  spaces  in connection  with  the  embedding  and  continuation theorems,
 Tr. Mat. Inst. Acad. Nauk SSSR  60 (1961)  42--81
 (in Russian).


\bibitem{Bugrov1964}
% 20
Ya.S. Bugrov,
 Approximation of a class of functions with
dominant mixed derivative,
Sb. Math. (N.S.) 106 (3) (1964) 410--418 (in Russian).

%  Я.С.Бугров.  Приближение  класса функций с  доминирующей  %смешанной производной. Мат.сб. - 1964. - {\bf 64}, № 3. - %С. 410 - 418.


 \bibitem{Carl1981} 
 % 48
 B. Carl,
 Entropy numbers, $s$-numbers, and eigen-value  problems,
 J.~Funct. Anal. 41 (1981) 290--306.

 \bibitem{Chen_Wang2017}
 J. Chen, H. Wang,
Preasymptotics and asymptotics of approximation numbers of anisotropic Sobolev embeddings,
J. Complexity
 39 (2017) 94--110.
%https://doi.org/10.1016/j.jco.2016.10.005

 \bibitem{Cobos_Kuhn_Sickel_2019}
F. Cobos, T. K\"{u}hn, W. Sickel,
On optimal approximation in periodic Besov spaces,
J. Math. Anal. Appl.
474 (2) (2019), 1441--1462.
%https://doi.org/10.1016/j.jmaa.2019.02.027.



\bibitem{Dung1984}
% 28
 D. D\~{u}ng,
 % Approximation  of  classes  of  functions  on the  torus prescribed by a mixed modulus  of  continuity.
%  In: Constructive theory  of functions (Proc. Internat. Conf., Varna, 1984), Publ. House Bulgarian Acad. Sci., Sofia, 1984,  43--48 (in Russian).
Approximation of classes of functions on the torus defined by a mixed modulus of continuity.
In: Constructive theory of functions (Proc. Internat. Conf., Varna, 1984). Bulgarian Academy of Science,
Sofia, 1984, 43--48 (in Russian).

\bibitem{Dung_Nguyen_2021}
D. D\~{u}ng, V. K. Nguyen,
High-dimensional nonlinear approximation by parametric manifolds
in H\"{o}lder-Nikol’skii spaces of mixed smoothness,
arXiv:2102.04370, 2021.

  \bibitem{Dung_Temlyakov_Ullrich2019} 
  % 11
 D. D\~{u}ng, V. Temlyakov,  T. Ullrich,
 Hyperbolic cross approximation,  Adv. Courses  Math. Birkh\"{a}user, CRM Barselona,  2018.
%doi:10.1007/978-3-319-92240-9

\bibitem{Dung_Ullrich2013}
D. D\~{u}ng, T. Ullrich, 
$N$-widths and $\varepsilon$-dimensions for high-dimensional approximations,
Found. Comput. Math. 13 (2013) 965--1003.
%https://doi.org/10.1007/s10208-013-9149-9

\bibitem{Fedunyk_Hembarskyi_Hembarska2020}
% 6
 O.V. Fedunyk-Yaremchuk, M.V. Hembars'kyi, S.B. Hembars'ka,
  Approximative  characteristics  of  the   Nikol'skii-Besov-type   classes  of periodic functions in the  space  $B_{\infty,1}$,
 Carpathian  Math. Publ.  12 (2) (2020) 376--391.

\bibitem{Fedunyk_Hembarska2022}
% 9
 O.V. Fedunyk-Yaremchuk, S.B. Hembars'ka,
 Best orthogonal  trigonometric approximations  of  the   Nikol'skii-Besov-type   classes  of periodic functions  of one and several variables,
 Carpathian  Math. Publ. 14 (1) (2022) 171--184.

\bibitem{Galeev1984}
% 29
  E.M. Galeev,
  Kolmogorov widths of  certain classes  of periodic  functions of several variables,
In: Constructive theory of functions (Proc. Internat. Conf., Varna, 1984). Bulgarian Academy of Science,
Sofia, 1984, 27--32 (in Russian).

   \bibitem{Galeev1990}
   % 31
 E.M. Galeev,
  Approximation of classes  of  periodic functions  of  several  variables  by  nuclear  operator,
 %  Mat. Zametki 47 (1990) 32--41.
Math. Notes 47 (3) (1990) 248--254.

\bibitem{Hembarska_Zaderey2022}
% 8
S.B. Hembars’ka, P.V. Zaderei,
Best orthogonal trigonometric approximations of the Nikol’skii–Besov-type classes of periodic functions in the space $B_{\infty, 1}$,
 Ukr. Math. J. 74 (6) (2022) 883--895. %https://doi.org/10.1007/s11253-022-02115-0


 \bibitem{Hollig}
 % 40
   K. H\"{o}llig,
   Diameters of classes  of  smooth  functions.
In:  Quantitative Approximation. Academic Press,   New York,  163--176.

\bibitem{Kashin_Temlyakov1994} 
% 2
B.S. Kashin, V.N. Temlyakov,
On best  $m$-term approximations and the entropy of sets in the space $L^1$,
 %   Mat. Zametki 56 (5) (1994) 57--86.
   Math. Notes 56 (5) (1994)  1137--1157.


\bibitem{Kashin_Temlyakov1999} 
% 44
 B.S. Kashin, V.N. Temlyakov,
 On a certain norm and  related applications,
 Math. Notes 64 (4) (1999) 551--554.

 \bibitem{Kashin_Temlyakov} 
 % 45
   B.S. Kashin, V.N. Temlyakov,
   On a norm  and approximation characteristics  of classes of  functions of  several variables.
 In:   Metric  theory  of functions  and  related  problems  in analysis.
 Izd. Nauchno-Issled. Aktuarno-Finans. Tsentra (AFTs),
     Moscow, 1999, 69--99 (in Russian).


\bibitem{Kolmogorov1936} 
% 36
  A. Kolmogoroff,
  \"{U}ber die beste Ann\"{a}herung von Fuktionen einer  gegebenen  Funktionenklasse,
   Ann.  Math.  37 (1) (1936) 107--110.
 %  doi:10.2307/1968691

    \bibitem{Kolmogorov_Tikhomirov1959}
    % 39
    A.N. Kolmogorov, V.M. Tikhomirov,
    $\varepsilon$-entropy  and  $\varepsilon$-capacity of sets in  functional spaces,
    Uspekhi Mat. Nauk, 14 (2) (1959) 3--86 (in~Russian).

\bibitem{Kuehn_Mayer_Ullrich_2016}
T. K\"{u}hn, S. Mayer, T. Ullrich, 
Counting via entropy: new preasymptotics for the approximation numbers of Sobolev embeddings,
SIAM J. Num. Anal., 54 (6) (2016) 3625--3647.
%https://doi.org/10.1137/16M106580X

 \bibitem{Kuehn_Sickel_Ullrich2015}
T. K\"{u}hn, W. Sickel, T. Ullrich, 
Approximation of mixed order Sobolev functions on the $d$-torus: asymptotics, preasymptotics, and $d$-dependence,
Constr. Approx., 42 (2015) 353--398. 
%https://doi.org/10.1007/s00365-015-9299-x

 \bibitem{Kuehn_Sickel_Ullrich2021}
T. K\"{u}hn, W. Sickel, T. Ullrich, 
How anisotropic mixed smoothness affects the decay of singular numbers for Sobolev embeddings,
J. Complexity, 63 (2021) 101523.
%https://doi.org/10.1016/j.jco.2020.101523


\bibitem{Lizorkin_Nikol'skii1989}
% 13
  P.I. Lizorkin, S.M. Nikol'skii,
 % Function spaces of  mixed  smoothness from  the decomposition point  of  view, %  Tr. Mat. Inst. Steklova
Spaces of functions of mixed smoothness from the decomposition point of view,
 Tr. Mat. Inst. Acad. Nauk SSSR 187 (3) (1989) 143--161 (in Russian).

 \bibitem{Mayer_Ullrich2020}
 % 42
   S. Mayer, T. Ullrich,
   Entropy  numbers of finite dimensional mixed-norm balls  and function space embeddings with small mixed smoothness,
    Constr. Approx. 53 (2021) 249--279.
    %https://doi.org/10.1007/s00365-020-09510-5

 \bibitem{Nikol'skii1951} 
 % 17
 S.M. Nikol'skii,
 Inequalities   for  entire  functions  of  finite power  and  their  application to the theory  of  differentiable   functions  of  many  variables,
  Tr. Mat. Inst. Akad. Nauk SSSR 38 (1951) 244--278
  (in Russian).

\bibitem{Nikol'skii1963}
% 14
 S.M. Nikol'skii,
  Functions  with dominant  mixed  derivative, satisfying a multiple Holder  condition,
  Sibirsk. Mat. Zh. 4 (6) (1963) 1342--1364 (in Russian).

\bibitem{Nikolskaja1974} 
% 21
% Н.С.Никольская,
 %  Приближение дифференцируемых   функций  многих  переменных  суммами Фурье в метрике $L_p$,
 %   Сиб.мат.журн. - 1974. - 15, № 2. - С.395 - 412.
N.S. Nikol’skaya,
The approximation in the $L_p$ metric of differentiable functions of several variables by Fourier sums,
Sib.
%erian
Math. J. 15 (2) (1974) 282--295.
%https://doi.org/10.1007/BF00968291

\bibitem{Romanyuk1991}
% 22
   A.S. Romanyuk,
  Approximation of the Besov classes of periodic functions of several variables in a space $L_q$,
   Ukr. Math. J. 43 (10) (1991) 1297–1306.
%https://doi.org/10.1007/BF01061817
 %   А.С. Романюк.  Приближение классов Бесова периодических   функций многих переменных в пространстве $L_q$.
  %Укр. мат.~журн.- 1991. - 43, № 10. - С.1398 - 1408.

 \bibitem{Romanyuk2004}
 % 25
  A.S. Romanyuk,
Approximability of the classes  $B^{\boldsymbol{r}}_{p,\theta}$ of  periodic  functions  of several   variables by  linear  methods and best approximations,
Sb.  Math. 195 (2) (2004) 237--261.
% https://doi.org/10.1070/SM2004v195n02ABEH000801
%Mat. Sbornik, 2004, Volume 195, Number 2, Pages 91–116
%DOI: https://doi.org/10.4213/sm801


\bibitem{Romanyuk2005_UMB} 
% 51
  A.S. Romanyuk,
  Kolmogorov widths of Besov classes $B^{\boldsymbol{r}}_{p,\theta}$  in the  metric  of the   space  $L_{\infty}$,
 %  Ukr. Mat. Visn. 2 (2005) 201--218.
  % English  translation in
  Ukr. Math. Bull., 2 (2) (2005) 205--222.
% ??? Не можу знайти посилання на англійську версію

\bibitem{Romanyuk2008}
% 32
    A.S. Romanyuk,
    Best approximations   and widths of classes  of  periodic  functions  of several   variables,
  Sb.  Math. 199 (2) (2008) 253--275.
  % Rus. version 93--114.
% https://doi.org/10.1070/SM2008v199n02ABEH003918

 \bibitem{Romanyuk2011}
 % 26
   A.S. Romanyuk,
   Diameters and best approximation   of the  classes
    $B^{\boldsymbol{r}}_{p,\theta}$ of  periodic  functions  of several   variables,
    Anal. Math. 37 (2011) 181--213.

\bibitem{Romanyuk2012}
% 18
  A.S. Romanyuk,
    Approximative charecteristics  of the  classes  of  periodic  functions  of  many variables.
 In:   Proc. of the Institute of Mathematics, National Academy of Sciences  of Ukraine, Kyiv, Vol. 93, 2012 (in Russian).


\bibitem{Romanyuk2015}
% 34
  A.S. Romanyuk,
 Estimation of the entropy numbers and Kolmogorov widths for the Nikol’skii–Besov classes of periodic functions of many variables,
  Ukr. Math. J. 67 (11) (2016) 1739--1757.
  % https://doi.org/10.1007/s11253-016-1186-5
%  А.С. Романюк. Оценки энтропийных чисел и колмогоровских поперечников   классов  Никольского-Бесова периодических   функций многих %переменных. Укр.~мат.~журн. 2015.  {\bf 67} (11), 1540-1556.

   \bibitem{Romanyuk2019}
   % 41
   A.S. Romanyuk,
   Entropy numbers and widths for the Nikol'skii-Besov   classes   of   functions  of many    variables in the space $L_{\infty}$,
   Anal. Math. 45 (1) (2019) 133--151.

\bibitem{Romanyuk_Romanyuk2020} 
% 4
  A.S. Romanyuk, V.S. Romanyuk,
  Estimation  of  some  approximating  characteristics  of  the  classes  of   periodic  functions  of  one  and  many  variables,
  Ukr.  Math. J. 71 (8) (2020) 1257--1272.
  %(translation  of Ukrain.Mat. Zh. {\bf 71} (8) (2019) 1102 - 1115 ( in Ukrainian)
%https://doi.org/10.1007/s11253-019-01711-x

\bibitem{Romanyuk_Romanyuk2021}
% 5
  A.S. Romanyuk, V.S. Romanyuk,
  Approximative  characteristics  and properties of  operators of  the best approximation of   classes  of    functions from  the Sobolev and Nikol'skii-Besov spaces,
   J. Math. Sci. 252 (4) (2021) 508--525.
% https://doi.org/10.1007/s10958-020-05177-2

   \bibitem{Romanyuk_Romanyuk_Pozharska_Hembarska2023}
   % 35
  A.S. Romanyuk, V.S. Romanyuk, K.V. Pozharska,  S.B. Hembars'ka,
  Characteristics of linear and nonlinear  approximation  of isotropic   classes  of periodic  multivariate functions,
  Carpathian  Math. Publ.  15 (1) (2023) 78--94.


 \bibitem{Romanyuk_Yanchenko2022__UMZ_74}
 % 46
 A.S. Romanyuk, S.Ya. Yanchenko,
 Kolmogorov widths of the Nikol’skii–Besov classes of periodic functions of many variables in the space of quasicontinuous functions,
 Ukr. Math. J. 74 (2) (2022) 251--265.
 % https://doi.org/10.1007/s11253-022-02061-x

%    Колмогоровські  поперечники  класів Нікольського - Бєсова  періодичних  функцій  багатьох змінних у просторі квазінеперервних  функцій.
%    Укр.мат.журн.
 %   2022. т.74, № 2. - С. 220-232.


\bibitem{Romanyuk_Yanchenko2022}
% 7
 A.S. Romanyuk, S.Ya. Yanchenko,
  Approximation  of  the classes of  periodic  functions  of  one  and many  variables  from  the  Nikol'skii-Besov and Sobolev spaces,
    Ukr. Math. J. 74 (6) (2022) 967–980.
 %https://doi.org/10.1007/s11253-022-02110-5
  %(translation  of Ukrain. Math.Zh. 74 (6) (2022) 857--868 ( in Ukrainian)).

 \bibitem{Romanyuk_Yanchenko2023}
 % 47
   A.S. Romanyuk, S.Ya. Yanchenko,
   Estimates for  the  entropy numbers  of the   Nikol'skii-Besov classes of   functions  with  mixed  smoothness  in  the space  of  quasi-continuons functions,
   Math. Nachr. 296 (6) (2023) 2575--2587.


\bibitem{Temlyakov1980} 
% 24
V.N. Temlyakov,
  Approximation of   periodic functions   of  several variables with bounded mixed  difference,
 Sb. Mat. 41 (1) (1982)  53--66.
% https://doi.org/10.1070/SM1982v041n01ABEH002220
% Mathematics of the USSR-Sbornik, 1982, Volume 41, Issue 1, Pages 53–66
%Matematicheskii Sbornik. Novaya Seriya, 1980, Volume 113(155), Number 1(9), Pages 65–80



\bibitem{Temlyakov1980.1}
% 27
  V.N. Temlyakov,
  On the approximation of   periodic functions   of  several variables with bounded mixed  difference,
   Dokl. AN SSSR 253 (1980)  544--548;
   English transl. in
   Soviet Math. Dokl.  22 (1980).
%https://www.mathnet.ru/php/archive.phtml?wshow=paper&jrnid=dan&paperid=43752&option_lang=eng

   \bibitem{Temlyakov1983}
    % 23
  V.N. Temlyakov,
  Approximation of  functions with a bounded mixed  difference   by  trigonometric polynomials, and the widths  of  some  classes  of  functions,
  Math. USSR Izv. 20 (1) (1983) 173--187.
% https://doi.org/10.1070/IM1983v020n01ABEH001346
 % Izvestia AN SSSR, Ser. Mat. 46 (1982)  171--186;
 

\bibitem{Temlyakov1989_121p}
% 12
 V.N. Temlyakov,
 Approximation of  functions with bounded mixed derivative,
 Proc. Steklov Inst. Math. 178 (1989) 1--121.
%https://www.mathnet.ru/php/archive.phtml?wshow=paper&jrnid=tm&paperid=2107&option_lang=eng

\bibitem{Temlyakov1989} 
% 30
  V.N. Temlyakov,
    Estimates of best bilinear  approximations of   periodic functions,
Proc. Steklov Inst. Math. 4 (1989) 275--293.
 %   Trudy MIAN  181 (1988)  250--267;
 

\bibitem{Temlyakov1990}
% 1
  V.N. Temlyakov,
   Estimates of the asymptotic  characteristics of classes  of  functions  with  bounded mixed derivative or difference,
  Proc. Steklov Inst. Math. 189  (1990) 161--197.
  % (translation of Tr. Mat. Inst.Steklova {\bf 189}  (1989)  138 - 168 (in Russian).

  \bibitem{Temlyakov1993}
  % 19
 V.N. Temlyakov,  
  Approximation of periodic function,   Nova Science Publishes, Inc., New York, 1993.


 \bibitem{Temlyakov1995_Complexity} % 50
 V.N. Temlyakov,
An inequality for trigonometric polynomials and its application for estimating the entropy numbers,
J. Complexity 11 (2) (1995) 293--307.
%https://doi.org/10.1006/jcom.1995.1012.

 \bibitem{Temlyakov1996_East_J}
 % 49
 V.N. Temlyakov,
  An inequality for  trigonometric  polynomials  and its application  for  estimating  the  Kolmogorov widths,
   East J. Approx. 2 (1996) 253–262.
   %89--98.
   

 \bibitem{Temlyakov2018} 
 % 10
 V.N. Temlyakov,
 Multivariate  approximation,  Cambridge University  Press, 2018.


 \bibitem{Temlyakov_Ullrich2020}
 % 43
   V. Temlyakov, T. Ullrich,
   Approximation of functions with small mixed smoothness in the  uniform norm,
   J. Approx. Theory  277 (2022) 105718.
% https://doi.org/10.1016/j.jat.2022.105718



\bibitem{Tikhomirov1960}
% 37
  V.M. Tikhomirov,
  % Widths  of sets  in functional spaces and  the  theory  of  best  approximations,
%   Uspekhi Mat. Nauk. 15 (1960), 81--120 (in Russian).
Diameters of sets in function spaces and the theory of best approximations,
Russian Math. Surveys 15 (3) (1960) 75--111.
% https://doi.org/10.1070/RM1960v015n03ABEH004093.

 \bibitem{Trigub_Belinsky2004}
 % 38
    R.M. Trigub, E.S. Belinsky,
    Fourier  Analysis  and  Approximation  of  Functions,   Kluwer Academic Publishers, 
    %Dordrecht,  
    2004.


\bibitem{Vybiral2006}
J. Vybiral,
Function spaces with dominating mixed smoothness, 
Dissertationes Mathematicae 436 (2006) 1-73. 
% DOI: 10.4064/dm436-0-1













\end{thebibliography}


\end{document}